\documentclass[10pt]{amsart}
\usepackage{bm}

\usepackage{amsmath}
\usepackage{amsthm}
\usepackage{amssymb}

\usepackage{hyperref}

\newcommand{\R}{\mathbb{R}}

\newcommand{\N}{\mathbb{N}}
\newcommand{\Z}{\mathbb{Z}}
\newcommand{\lr}[1]{\langle #1 \rangle}
\newcommand{\Lr}[1]{\Big\langle #1 \Big\rangle}
\newcommand{\eps}{\varepsilon}
\newcommand{\wt}[1]{\widetilde{#1}}
\newcommand{\wh}[1]{\widehat{#1}}

\newcommand{\Sc}{\mathcal{S}}
\newcommand{\Lc}{\mathcal{L}}
\newcommand{\J}{\mathcal{J}}

\newcommand{\vr}{\varrho}

\newcommand{\ol}[1]{{\overline{#1}}}

\newcommand{\uh}{u^{\rm{hyp}}}
\newcommand{\uhp}{u^{\rm{hyp},+}}

\newcommand{\ue}{u^{\rm{ell}}}
\newcommand{\uep}{u^{\rm{ell},+}}
\newcommand{\PN}{P_{\frac{N}{2^{\delta}} \le \cdot \le 2^{\delta}N}}

\newcommand{\dx}{\partial_x}
\newcommand{\dy}{\partial_y}
\newcommand{\dt}{\partial_t}

\newcommand{\dz}{\partial_{\zeta}}
\newcommand{\de}{\partial_{\eta}}

\newcommand{\err}{\bm{\mathrm{err}}}

\newcommand{\Or}{\Omega_{\rho} (t)}
\newcommand{\Oo}{\Omega_0 (t)}
\newcommand{\Orz}{\Omega_{\rho}^0 (t)}

\newcommand{\Xp}{\mathfrak{X}^+ (t)}
\newcommand{\Xn}{\mathfrak{X}^- (t)}
\newcommand{\Xz}{\mathfrak{X}^0 (t)}
\newcommand{\Xf}{\wh{\mathfrak{X}}^- (t)}

\newtheorem{thm}{Theorem}[section]
\newtheorem{prop}[thm]{Proposition}

\newtheorem{lem}[thm]{Lemma}
\newtheorem{cor}[thm]{Corollary}

\theoremstyle{remark}
\newtheorem{rmk}[thm]{Remark}

\DeclareMathOperator{\supp}{supp}
\allowdisplaybreaks[3]

\numberwithin{equation}{section}

\date{\today}

\title[Long-time behavior of solutions to 5mKdV-type equation]{Long-time behavior of solutions to the fifth-order modified KdV-type equation}

\author[M. Okamoto]{Mamoru Okamoto}
\address{Division of Mathematics and Physics, Faculty of Engineering, Shinshu University, 4-17-1 Wakasato, Nagano City 380-8553, Japan}
\email{m\_okamoto@shinshu-u.ac.jp}



\date{\today}

\begin{document}

\begin{abstract}
We consider the long-time behavior of solutions to the fifth-order modified KdV-type equation.
Using the method of testing by wave packets, we prove the small-data global existence and modified scattering.
We derive the leading asymptotic in both the self-similar and oscillatory regions.
\end{abstract}

\maketitle

\section{Introduction}

We consider the Cauchy problem for the following equation of fifth-order modified Korteweg-de Vries (mKdV) type:
\begin{equation} \label{5mKdVg}
\begin{aligned}
& \dt u - \frac{1}{5} \dx^5 u = \dx \left( c_1 u^2 \dx^2 u + c_2 u (\dx u)^2 + c_3 u^5 \right), \\
& u(0,x) = u_0(x) ,
\end{aligned}
\end{equation}
where $u = u(t,x) : \R \times \R \rightarrow \R$ is an unknown function and $u_0$ is a given function.
Here, $c_1,c_2, c_3$ are real constants.
The equation with $c_1=c_2= -2$ and $c_3=\frac{6}{5}$ is known as the fifth-order mKdV equation, which is the second equation from the mKdV hierarchy:
\begin{equation} \label{5mKdVc}
\dt u - \frac{1}{5} \dx^5 u = 6 u^4 \dx u - \dx \left(u \dx^2 (u^2) \right).
\end{equation}
This equation is completely integrable in the sense that there are Lax pair formulations, and enjoy infinite number of conservation quantities.

Well-posedness of the Cauchy problem for \eqref{5mKdVg} has been well studied.
Kenig et al. \cite{KPV94} studied the local-in-time well-posedness of the higher-order KdV-type equations:
\[
\dt u + \dx^{2j+1} u + P(u, \dx u, \dots, \dx^{2j} u)=0,
\]
where $P$ is a polynomial having no constant and linear terms.
They proved local well-posedness for the initial data in the weighted Sobolev space $H^{s,m} (\R) := H^s(\R) \cap L^2 (|x|^{2m} dx)$ for some (large) $s, m \ge 0$.
Kwon \cite{Kwo08} proved local well-posedness for \eqref{5mKdVg} in $H^s (\R)$ with $s \ge \frac{3}{4}$.
Moreover, the flow map of the fifth-order mKdV-type equations fails to be uniformly continuous for $s < \frac{3}{4}$.
This implies that the regularity $\frac{3}{4}$ is the minimal regularity threshold for which the well-posedness result can be solved via an iteration method.
Gr\"{u}nrock \cite{Gru10} investigated well-posedness for the higher-order mKdV equations in $\mathcal{F} L^p (\R)$-based spaces.

In the case of the mKdV equation, well-posedness has been extensively studied (see \cite{KPV89, KPV93, CKSTT03, Guo09, Kis09} and the references therein).
Because the mKdV equation is complete integrable, inverse scattering techniques as in Deift and Zhou \cite{DeiZho93} show global existence and asymptotic behavior.
Hayashi and Naumkin \cite{HayNau99, HayNau01} derived modified asymptotics without relying on complete integrability.
By developing the factorization technique, in \cite{HayNau16} they improved the previous result.
Harrop-Griffiths \cite{HarG16} proved the long-time behavior of solutions to the mKdV equation.
His result used the method of testing by wave packets, developed in the work of Ifrim and Tataru \cite{IfrTat15} on the cubic nonlinear Schr\"{o}dinger equation (see also \cite{IfrTat16, HIT}).
This method in some sense interpolates between the physical and the Fourier side analysis.
In this paper, we employ the method of testing by wave packets to show global existence and asymptotic behavior of solutions to the fifth-order mKdV-type equation.

Hayashi and Naumkin \cite{HayNau98} proved that the solution to the generalized KdV equation
\[
\dt u - \frac{1}{3} \dx^3 u = u^p \dx u
\]
is asymptotically free for $p>3$, namely that there exists a linear solution $v^+$ such that $u(t) \to v^+(t)$ in $L^2 (\R)$ as $t \to \infty$.
On the other hand, asymptotic behavior of the solution to the mKdV equation ($p=3$) differs from that of the linear solutions.
Hence, we call the nonlinearity of the mKdV equation critical in the sense of the large-time behavior.
As in the mKdV equation, while a solution $u$ to \eqref{5mKdVg} exists globally, we expect the asymptotic behavior of $u$ to differ from that of the linear solutions.
To explain this phenomenon, we roughly derive the asymptotic behavior of linear solutions.
We note that the linear solution is written as follows:
\[
e^{\frac{1}{5} t \dx^5} u_0 (x)
= (\mathcal{F}^{-1} [e^{\frac{1}{5}it \xi^5}] \ast u) (x), \quad
\mathcal{F}^{-1} [e^{\frac{1}{5}it \xi^5}] = \frac{1}{\sqrt{2\pi}} \int_{\R} e^{i(x\xi+\frac{1}{5}t \xi^5)} d\xi.
\]
Because $\partial_{\xi} (x\xi+\frac{1}{5}t \xi^5) = x + t \xi^4$ becomes zero if and only if $\xi = \pm (\frac{|x|}{t})^{\frac{1}{4}}$ and $x<0$, the stationary phase method implies that the linear solution $e^{\frac{1}{5} t \dx^5}u_0 (x)$ decays rapidly as $t^{-\frac{1}{5}} x \to +\infty$ and oscillates as $t^{-\frac{1}{5}} x \to -\infty$.
Moreover, in the self-similar region $t^{-\frac{1}{5}} |x| \lesssim 1$, we have
\[
e^{\frac{1}{5} t \dx^5}u_0 (x)
= t^{-\frac{1}{5}} Q_0 (t^{-\frac{1}{5}}x) \int_{\R} u_0 (y) dy + \text{error},
\]
where $Q_0$ is a solution to $Q_0'''' + y Q_0=0$.
In the oscillatory region $t^{-\frac{1}{5}} x \to -\infty$, there exists a constant $c$ such that
\[
u(t) = c t^{-\frac{1}{5}} (t^{-\frac{1}{5}} |x|)^{-\frac{3}{8}} \Re \left( \wh{u_0} (t^{-\frac{1}{4}} |x|^{\frac{1}{4}}) e^{i \phi (t,x)} \right) + \text{error},
\]
where the phase function is given by
\begin{equation} \label{phase}
\phi (t,x) = -\frac{4}{5} t^{-\frac{1}{4}} |x|^{\frac{5}{4}} + \frac{\pi}{4}.
\end{equation}
This observation implies that for smooth initial data with $\| u_0 \|_{H^{0,1}} \le \eps$ and $k=0,1,2,3$, we have
\[
| \dx^k u(t,x) | \lesssim \eps t^{-\frac{k+1}{5}} \lr{t^{-\frac{1}{5}}x}^{\frac{k}{4}-\frac{3}{8}}.
\]
In particular, $| u \dx^3 u | + | \dx u \dx^2 u | + |u^3 \dx u| \lesssim \eps^2 t^{-1}$ holds true.
We expect solutions to \eqref{5mKdVg} to have the same pointwise estimates as linear solutions above, and hence this bound causes critical phenomena in the sense of the large-time behavior.

Setting $(c_1,c_2) = \alpha (2,3)$ and $c_3= \beta$ for real constants $\alpha$ and $\beta$, we focus on the following Cauchy problem:
\begin{equation} \label{5mKdV}
\begin{aligned}
& \dt u - \frac{1}{5} \dx^5 u = \dx \left( \alpha (2 u^2 \dx^2 u + 3 u (\dx u)^2) + \beta u^5 \right), \\
& u(0,x) = u_0(x).
\end{aligned}
\end{equation}
This nonlinearity cancels out a part that is difficult to handle in the energy estimate.
Although \eqref{5mKdV} does not include the fifth-order mKdV equation \eqref{5mKdVc}, this is the first result of asymptotic behavior for the fifth-order mKdV-type equation with critical nonlinearity in the sense of the large-time behavior.

\begin{thm} \label{thm}
Assume that the initial datum $u_0$ at time $0$ satisfies
\[
\| u_0 \|_{H^{2,1}} \le \eps \ll 1.
\]
Then, there exists a unique global solution $u$ to \eqref{5mKdV} with $e^{-\frac{1}{5} t \dx^5} u_0 \in C(\R; H^{2,1}(\R))$ satisfying the estimates
\begin{equation} \label{est:u_infty}
\| \lr{t^{-\frac{1}{5}}x}^{-\frac{k}{4}+\frac{3}{8}} \dx^k u(t) \|_{L^{\infty}} \lesssim \eps t^{-\frac{k+1}{5}}
\end{equation}
for $t \ge 1$ and $k=0,1,2,3$.
Moreover, we have the following asymptotic behavior as $t \to +\infty$.

In the decaying region $\Xp := \{ x \in \R_+ \colon t^{-\frac{1}{5}} |x| \gtrsim t^{\frac{4}{5}(\frac{1}{10}-\eps)} \}$, we have
\[
\| t^{\frac{1}{5}} \lr{t^{-\frac{1}{5}} x}^{\frac{7}{8}} u \|_{L^{\infty} (\Xp)} \lesssim \eps, \quad
\| t^{\frac{1}{10}} \lr{t^{-\frac{1}{5}} x} u \|_{L^2 (\Xp)} \lesssim \eps.
\]
In the self-similar region $\Xz := \{ x \in \R \colon t^{-\frac{1}{5}} |x| \lesssim t^{\frac{4}{5}(\frac{1}{10}-\eps)} \}$, there exists a solution $Q =Q(y)$ to the nonlinear ordinary differential equation
\begin{equation} \label{eq:self}
Q'''' + y Q + 5\alpha (2 Q^2 Q'' + 3 Q (Q')^2) + 5\beta Q^5 =0,
\end{equation}
satisfying $\| Q \|_{L^{\infty}_y} \lesssim \eps$ and we have the estimates
\begin{align*}
& \| u(t) - t^{-\frac{1}{5}} Q(t^{-\frac{1}{5}} x) \|_{L^{\infty} (\Xz)} \lesssim \eps t^{-\frac{7}{10} (\frac{27}{70} - \eps)}, \\
& \| u(t) - t^{-\frac{1}{5}} Q(t^{-\frac{1}{5}} x) \|_{L^2 (\Xz)} \lesssim \eps t^{-\frac{4}{5} (\frac{9}{40} - \eps)}.
\end{align*}
In the oscillatory region $\Xn := \{ x \in \R_- \colon t^{-\frac{1}{5}} |x| \gtrsim t^{\frac{4}{5}(\frac{1}{10}-\eps)} \}$, there exists a unique (complex-valued) function $W$ satisfying $W(\xi) = \ol{W(-\xi)}$ and $\| W \|_{L^{\infty}} \lesssim \eps$ such that
\begin{align*}
u(t,x) = & t^{-\frac{1}{5}} (t^{-\frac{1}{5}} |x|)^{-\frac{3}{8}} \Re \left\{ W \left( t^{-\frac{1}{4}} |x|^{\frac{1}{4}} \right) e^{i \phi (t,x) - \frac{3}{4}i \alpha | W ( t^{-\frac{1}{4}} |x|^{\frac{1}{4}} ) |^2 \log (t^{-\frac{1}{4}} |x|^{\frac{5}{4}})} \right\}
\\ & +\err_x,
\end{align*}
where the error satisfies the estimates
\[
\| t^{\frac{1}{5}} (t^{-\frac{1}{5}} |x|)^{\frac{9}{16}} \err_x \|_{L^{\infty} (\Xn)} \lesssim \eps, \quad
\| t^{\frac{1}{10}} (t^{-\frac{1}{5}} |x|)^{\frac{3}{8}} \err_x \|_{L^2 (\Xn)} \lesssim \eps.
\]
In the corresponding frequency region $\Xf := \{ \xi \in \R \colon t^{\frac{1}{5}} |\xi| \gtrsim t^{\frac{1}{5} (\frac{1}{10}-\eps)} \}$, we have
\[
\wh{u} (t,\xi) = W(\xi) e^{\frac{1}{5}it\xi^5 - \frac{3}{4}i \alpha |W(\xi)|^2 \log (t\xi^5)} + \err_{\xi},
\]
where the error satisfies
\[
\| (t^{\frac{1}{5}} \xi)^{\frac{3}{4}} \err_{\xi} \|_{L^{\infty} (\Xf)} \lesssim \eps, \quad
\| t^{\frac{1}{10}} (t^{\frac{1}{5}} \xi)^{\frac{3}{2}} \err_{\xi} \|_{L^2 (\Xf)} \lesssim \eps.
\]
\end{thm}

By taking the transformation $u (t,x) \mapsto u(-t,-x)$, we obtain the corresponding asymptotic behavior as $t \to -\infty$.

We make some remarks.
For the local-in-time well-posedness, the assumption $u_0 \in H^2(\R)$ is not needed:
in fact, we show that $u_0 \in H^{\frac{3}{4}} (\R)$ is enough in Proposition \ref{prop:WP}.
However, we require more regularity to obtain a global solution (see Remark \ref{rmk:reqH2}).

The large-time asymptotics of solutions in the oscillatory region $\Xn$ have a logarithmic correction in the phase comparing with the corresponding linear case.
This correction depends only on the cubic nonlinearity and hence \eqref{5mKdV} with $\alpha=0$ behaves like the linear solutions in this region.
The function $W$ does not belong to $L^2(\R)$ in general because solutions to \eqref{5mKdV} do not lead to conservation of the $L^2$-norm.
We have only an a priori bound $\| u(t) \|_{H^2} \lesssim \eps \lr{t}^{C \eps}$, which is enough to obtain a global solution.

If $\alpha=0$, then \eqref{5mKdV} is written as
\begin{equation} \label{5mKdV5}
\dt u - \frac{1}{5} \dx^5 u  = \beta \dx (u^5),
\end{equation}
which has a Hamiltonian structure.
In particular,
\[
\| u(t) \|_{L^2}, \quad
\int_{\R} \left( \frac{1}{2} |\dx^2 u(t,x)|^2 + \frac{5}{6} \beta |u(t,x)|^6 \right) dx
\]
are conserved, i.e., independent of $t$ as long as $u$ is a solution to \eqref{5mKdV5}.
The second quantity is the energy, hence the energy space is $H^2(\R)$.
The $L^2$-conservation law yields that $\| W \|_{L^2} \lesssim \eps$ provided that $\alpha=0$.
Moreover, we can replace $u_0 \in H^2(\R)$ with $u_0 \in H^{\frac{2}{5}}(\R)$.

\begin{thm} \label{thm2}
Assume that the initial datum $u_0$ at time $0$ satisfies
\[
\| u_0 \|_{H^{\frac{2}{5},1}} \le \eps \ll 1.
\]
Then, there exists a unique global solution $u$ to \eqref{5mKdV5} with $e^{-\frac{1}{5} t\dx^5} u \in C(\R; H^{\frac{2}{5},1}(\R))$ which satisfies the estimates \eqref{est:u_infty}.
Moreover, we have the same asymptotic behavior as in Theorem \ref{thm} with $\alpha=0$.
In addition, the function $W$ also satisfies
\[
\| W \|_{L^{\infty} \cap L^2} \lesssim \eps.
\]
\end{thm}

We give here an outline of the proof.
Denote by $\Lc$ the linear operator of \eqref{5mKdV}:
\[
\Lc := \dt - \frac{1}{5} \dx^5.
\]
To obtain pointwise estimates for solutions, we use the vector field
\[
\J := x+t \dx^4,
\]
which satisfies $\J = e^{\frac{1}{5} t \dx^5} x e^{-\frac{1}{5} t \dx^5}$.

Equation \eqref{5mKdV} is invariant under the scaling transformation
\begin{equation*}
u (t,x) \mapsto \lambda u(\lambda^5 t, \lambda x)
\end{equation*}
for any $\lambda >0$.
The generator of the scaling transformation is given by
\[
\Sc := 5 t \partial_t + x \dx + 1 ,
\]
which is related to $\Lc$ and $\J$ as follows:
\[
\Sc = 5 t \Lc + \J \dx + 1.
\]
As in \cite{HarG16, HayNau16}, we also use the operator
\begin{equation}  \label{eq:deflambda}
\Lambda := \dx^{-1} \Sc = 5t \dx^{-1} \Lc + \J.
\end{equation}

We introduce the norm with respect to the spatial variable
\[
\| u (t) \|_{X^s} := \| u (t) \|_{H^s} + \| \Lambda u (t) \|_{L^2}
\]
for $s \in \R$.
We note that
\[
\| u_0 \|_{X^s} \sim \| u_0 \|_{H^{s,1}}.
\]
In \S \ref{wp}, by using the Fourier restriction norm method, we show local-in-time well-posedness of \eqref{5mKdV} in $C(\R;X^s)$.
We need to estimate nonlinear parts including $\Lambda u$.
To estimate the nonlinear parts unifiedly, we show regularity conditions whereby tri- and quinti-linear estimates hold in the Fourier restriction norm spaces.
This is an extension of Kwon's result \cite{Kwo08}.

The local well-posedness implies that for $\eps>0$ sufficiently small, we can find $T>1$ and a unique solution $u \in C([0,T]; X^s)$ to \eqref{5mKdV}.
We then make the bootstrap assumption that $u$ satisfies the linear pointwise estimates:
there exists a constant $D$ with $1<D \le \eps^{-\frac{1}{2}}$ such that
\begin{equation} \label{est:u_infty0}
\| \lr{t^{-\frac{1}{5}}x}^{-\frac{k}{4}+\frac{3}{8}} \dx^k u(t) \|_{L^{\infty}} \le D \eps t^{-\frac{k+1}{5}}
\end{equation}
for $t \in [1,T]$ and $k=0,1,2,3$.

In \S \ref{S:energy}, under this assumption, for $\eps>0$ sufficiently small, we have the energy estimate
\begin{equation*}
\sup_{0 \le t \le T} \| u(t) \|_{X^s} \le C_1 \eps \lr{T}^{C_2 \eps},
\end{equation*}
where $C_1$ and $C_2$ are constants independent of $D$ and $T$.
To complete the proof of global existence, we need to close the bootstrap estimate \eqref{est:u_infty0}.

Because the nonlinearity of \eqref{5mKdV} contains three derivatives (e.g., $u^2 \dx^3 u$), an unfavorable term appears in the standard energy argument.
To obtain the energy estimate for $\| u(t) \|_{H^s}$, we need to add some correction term (see \cite{Kwo09} for the fifth-order KdV equation, and see also \cite{KenPil16}).
We point out that this argument can also be applied to more general nonlinearity, namely \eqref{5mKdVg}.
On the other hand, we rely on the special nonlinearity of \eqref{5mKdV} to obtain the energy estimate for $\Lambda u$.
More precisely, the nonlinearity of \eqref{5mKdV} ensures that $\int_{\R} u \dx u (\dx \Lambda u)^2 dx$ does not appear in $\dt \| \Lambda u (t) \|_{L^2}^2$.

In \S \ref{S:KS_type}, we prove a priori bounds that give the pointwise and $L^2$ decay estimates.
In particular, the estimate \eqref{est:u_infty} in the decaying region $\Xp$ holds true.
We also observe that \eqref{est:u_infty0} holds true at $t=1$.

To obtain \eqref{est:u_infty} in the self-similar region $\Xz$ or the oscillatory region $\Xn$, we use the method of testing by wave packets as in \cite{HarG16, HIT, IfrTat15, O}.
In \S \ref{S:wave_packet}, we observe that the wave packet is a good approximate solution to the linear fifth-order mKdV equation.
We also show that the output $\gamma$ of testing solutions $u$ to \eqref{5mKdV} with the wave packet is a reasonable approximation of $u$.
We then reduce closing the bootstrap estimate \eqref{est:u_infty0} to proving global bounds for $\gamma$.

In \S \ref{S:proof}, combining the estimates proved in previous sections, we show that $\gamma$ satisfies an ordinary differential equation.
Solving this ordinary differential equation shows the global bounds for $\gamma$, which concludes the proof, and the logarithmic correction to the phase arises.

At this point of this section, we summarize the notation used throughout this paper.
We set $\N_0 := \N \cup \{ 0 \}$.
We denote the set of positive and negative real numbers by $\R_+$ and $\R_-$, respectively.
We denote the space of all smooth and compactly supported functions on $\R$ by $C_0^{\infty} (\R)$.
We denote the space of all rapidly decaying functions on $\R$ by $\mathcal{S}(\R)$.
We define the Fourier transform of $f$ by $\mathcal{F}[f]$ or $\widehat{f}$.
We denote the inhomogeneous Sobolev spaces by $H^s(\R)$ equipped with the norm $\| f \|_{H^s}:= \| \lr{\cdot}^s \widehat{f} \| _{L^2}$, where $\lr{\xi} := (1+ | \xi |^2)^{\frac{1}{2}}$.
We also denote the homogeneous Sobolev spaces by $\dot{H}^s(\R)$.
We define the weighted Sobolev norms by $\| f \|_{H^{s,m}} := \| f \|_{H^s} + \| |x|^m f \|_{L^2}$.

In estimates, we use $C$ to denote a positive constant that can change from line to line.
If $C$ is absolute or depends only on parameters that are considered fixed, then we often write $X \lesssim Y$, which means $X \le CY$.
When an implicit constant depends on a parameter $a$, we sometimes write $X \lesssim_{a} Y$.
We define $X \ll Y$ to mean $X \le C^{-1} Y$ and $X \sim Y$ to mean $C^{-1} Y \le X \le C Y$.
We write $X = Y + O(Z)$ when $|X-Y| \lesssim Z$.

Let $\delta>0$ be a small constant, which is needed only to demonstrate Proposition \ref{prop:gamma_decay}.
For concreteness, we take $\delta = \frac{1}{100}$.
Let $\sigma \in C_0^{\infty}(\R )$ be an even function with $0 \le \sigma \le 1$ and $\sigma (\xi ) = \begin{cases} 1, & \text{if } |\xi | \le 1, \\ 0, & \text{if } |\xi| \ge 2^{\delta} .\end{cases}$
For any $R, \, R_1, R_2 >0$ with $R_1 < R_2$, we set 
\begin{gather*}
\sigma _R (\xi) := \sigma \Big( \frac{\xi}{R} \Big) -\sigma \Big( \frac{2^{\delta} \xi}{R} \Big) , \quad \sigma _{\le R} (\xi) := \sigma \Big( \frac{\xi}{R} \Big) , \quad \sigma _{>R} (\xi) := 1- \sigma _{\le R} (\xi) ,\\
\sigma _{<R}(\xi) := \sigma_{\le R}(\xi) - \sigma _{R} (\xi) , \quad
\sigma _{\ge R}(\xi) := \sigma_{>R}(\xi) + \sigma _{R} (\xi), \\
\sigma_{R_1 \le \cdot \le R_2} (\xi) := \sigma _{\le R_2}(\xi) - \sigma_{<R_1}(\xi), \quad
\sigma_{R_1 < \cdot < R_2} (\xi) := \sigma _{<R_2}(\xi) - \sigma_{\ge R_1}(\xi).
\end{gather*}
For any $N, \, N_1, N_2 \in 2^{\delta \mathbb{Z}}$ with $N_1 < N_2$, we define
\[
P_N f := \mathcal{F}^{-1} [\sigma_{N} \widehat{f}] , \quad
P_{N_1 \le \cdot \le N_2} f := \mathcal{F}^{-1}[\sigma_{N_1 \le \cdot \le N_2} \wh{f}] .
\]
We denote the characteristic function of an interval $I$ by $\bm{1}_{I}$.
For $N \in 2^{\delta \mathbb{Z}}$, we define the Fourier multipliers with the symbols $\bm{1}_{\R_{\pm}}(\xi)$ and $\sigma_N(\xi) \bm{1}_{\R_{\pm}} (\xi )$ by $P^{\pm}$ and $P_N^{\pm}$, respectively.

\section{Local well-posedness} \label{wp}

For $s,b \in \R$, we define the space $Y^{s,b}$ as the closure of $\mathcal{S} (\R \times \R)$ under the norm
\[
\| u \|_{Y^{s,b}} := \bigg\| \lr{\xi}^s \Lr{\tau-\frac{\xi^5}{5}}^b \mathcal{F}_{t,x} [u] \bigg\|_{L^2_{\tau, \xi}}.
\]
We note that $Y^{s,b} \hookrightarrow C(\R; H^s(\R)$ if $b>\frac{1}{2}$.
For our analysis, we need to introduce the local-in-time version of the space defined above.
When $b>\frac{1}{2}$, we define the corresponding restriction space $Y^{s,b}_T$ to a given time interval $[0,T]$ for $T>0$ as
\begin{equation*}
Y^{s,b}_T := \left\{ u \in L^{\infty}([0,T]; H^s(\R)) \colon \text{ there exists } v \in Y^{s,b} \text{ such that } \ v|_{[0,T]} = u \right\}.
\end{equation*}
We endow $Y^{s,b}_T$ with the norm 
\[
\| u \|_{Y^{s,b}_T} :=\inf \left\{ \| v \|_{Y^{s,b}} \colon \ v|_{[0,T]} = u \right\},
\]
where the infimum is taken over all possible extensions  $v$ of $u$ onto the real line.

Set
\[
\| u \|_{Z^{s,b}_T} := \| u \|_{Y^{s,b}_T} + \| \Lambda u \|_{Y^{0,b}_T}.
\]
In this section, we prove the following well-posedness result.

\begin{prop} \label{prop:WP}
Let $s \ge \frac{3}{4}$ and let $\vr>0$ be sufficiently small.
If $u_0 \in H^{s,1} (\R)$, then there exists $T=T( \| u_0 \|_{H^{\frac{3}{4},1}}) >0$ and a unique solution $u \in Z^{s,\frac{1}{2}+\vr}_T$ to \eqref{5mKdV} satisfying
\[
\sup _{0 \le t \le T} \| u(t) \|_{X^{s}} \lesssim \| u_0 \|_{H^{s,1}}.
\]
Moreover, the flow map $u_0 \in H^{s,1} (\R) \mapsto u \in Z^{s,\frac{1}{2}+\vr}_T$ is locally Lipschitz continuous.
\end{prop}

\begin{cor} \label{cor:WP}
Let $s \ge 0$  and let $\vr>0$ be sufficiently small.
If $u_0 \in H^{s,1} (\R)$, then there exists $T=T( \| u_0 \|_{H^{0,1}})>0$ and a unique solution $u \in Z^{s,\frac{1}{2}+\vr}_T$ to \eqref{5mKdV5} satisfying
\[
\sup _{0 \le t \le T} \| u(t) \|_{X^{s}} \lesssim \| u_0 \|_{H^{s,1}}.
\]
Moreover, the flow map $u_0 \in H^{s,1} (\R) \mapsto u \in Z^{s,\frac{1}{2}+\vr}_T$ is locally Lipschitz continuous.
\end{cor}

Using Proposition \ref{prop:WP} repeatedly, for sufficiently small $\| u_0 \|_{X^s}$, we can find an existence time $T>1$ and a unique solution $u \in C([0,T]; X^s)$ to \eqref{5mKdV}.

We list some useful properties of $Y^{s,b}$.
The following linear estimates are well-known (see for example \cite{Tao})

\begin{lem} \label{lem:Ylin}
Let $u$ satisfy
\[
\Lc u = F, \quad u(0,x) = u_0.
\]
For any $s \in \R$, $b>\frac{1}{2}$, and $0<T<1$, we have
\[
\| u \|_{Y^{s,b}_T} \lesssim \| u_0 \|_{H^s} + \| \sigma_{\le T}(t) F \|_{Y^{s,b-1}}.
\]
\end{lem}

\begin{lem} \label{lem:Ybb'}
Let $s \in \R$ and $-\frac{1}{2}<b' \le b < \frac{1}{2}$.
Then, for any $0<T<1$, we have
\[
\| \sigma_{\le T} (t) u \|_{Y^{s,b'}_T}
\lesssim T^{b-b'} \| u \|_{Y^{s,b}}.
\]
\end{lem}

Kenig et al. \cite{KPV91} showed the Strichartz estimates.

\begin{prop} \label{prop:Str}
Let $2 \le q, r \le \infty$ and $0 \le s \le \frac{3}{q}$ satisfy $-s+\frac{5}{q}+\frac{1}{r} = \frac{1}{2}$.
Then,
\[
\| |\dx|^s e^{\frac{1}{5} t \dx^5} u_0 \|_{L_t^q L_x^r}
\lesssim \| u_0 \|_{L^2},
\]
for all $u_0 \in L^2(\R)$.
\end{prop}

We use the bilinear estimate, which was proved by Kwon \cite{Kwo08}.

\begin{prop} \label{prop:bilinear}
For any $0< \vr \ll 1$, we have
\[
\| uv \|_{L^2}
\lesssim \| u \|_{Y^{-\frac{3}{4},\frac{1}{2}-\vr}} \| v \|_{Y^{0,\frac{1}{2}+\vr}}.
\]
\end{prop}

From the interpolation, we obtain the following.

\begin{lem} \label{lem:interpolation}
Let $0 < \vr \ll 1$.
Then, we have
\begin{equation} \label{interp1}
\| u \|_{L_{t,x}^6}
\lesssim \| u \|_{Y^{-\frac{1}{2}+10 \vr,\frac{1}{2}-\vr}}, \quad
\| u \|_{L_{t,x}^8}
\lesssim \| u \|_{Y^{-\frac{1}{4}+10\vr,\frac{1}{2}-\vr}}.
\end{equation}
Moreover, for $N_1, N_2 \in 2^{\delta \Z}$ with $N_1 \gg N_2$, we have
\begin{equation} \label{interp2}
\begin{aligned}
& \| P_{N_1}u P_{N_2}v \|_{L^2} \\
& \lesssim N_1^{-2+10\vr} \min \left( \| P_{N_1} u \|_{Y^{0,\frac{1}{2}-\vr}} \| P_{N_2} v \|_{Y^{0,\frac{1}{2}+\vr}}, \| P_{N_1} u \|_{Y^{0,\frac{1}{2}+\vr}} \| P_{N_2} v \|_{Y^{0,\frac{1}{2}-\vr}} \right)
\end{aligned}
\end{equation}
\end{lem}

\begin{proof}
Interpolating $Y^{0,0} = L_t^2(\R; L_x^2(\R))$ and $Y^{0,\frac{1}{2}+\vr_1} \hookrightarrow L_t^{\infty} (\R; L_x^2(\R))$ for any $0< \vr_1 \ll 1$, we have
\[
Y^{0,\frac{1}{3}+\frac{2}{3}\vr_1} \hookrightarrow L_t^6 (\R; L_x^2(\R)).
\]
Proposition \ref{prop:Str} and the transference principle imply
\begin{align*}
\| u \|_{L_{t,x}^6}
& \lesssim \| P_{\le 1} u \|_{L_t^6 L_x^2} + \| P_{\ge 1} u \|_{L_{t,x}^6}
\lesssim \| P_{\le 1} u \|_{Y^{0,\frac{1}{3}+\frac{2}{3}\vr_1}} + \| P_{\ge 1} u \|_{Y^{-\frac{1}{2},\frac{1}{2}+\vr_1}} \\
& \lesssim \| u \|_{Y^{-\frac{1}{2},\frac{1}{2}+\vr_1}}.
\end{align*}
From Sobolev's embedding $\dot{H}^{\frac{1}{3}} (\R) \hookrightarrow L^6(\R)$, we have
\[
\| u \|_{L_{t,x}^6}
\lesssim \| |\dx|^{\frac{1}{3}} u \|_{L_t^6 L_x^2}
\lesssim \| u \|_{Y^{\frac{1}{3}, \frac{1}{3}+\frac{2}{3}\vr_1}}.
\]
The interpolation $[Y^{-\frac{1}{2},\frac{1}{2}+\vr}, Y^{\frac{1}{3}, \frac{1}{3}+\vr}]_{12 \vr} = Y^{-\frac{1}{2}+10\vr, \frac{1}{2}-\vr}$ shows that
\[
\| u \|_{L_{t,x}^6}
\lesssim \| u \|_{Y^{-\frac{1}{2}+10\vr, \frac{1}{2}-\vr}}.
\]
Similarly, from embeddings $Y^{-\frac{1}{4}, \frac{1}{2}+\vr_1} \hookrightarrow L^8_{t,x}(\R^2)$, $Y^{\frac{3}{8},\frac{3}{8}+\frac{3}{4}\vr_1} \hookrightarrow L_{t,x}^8 (\R^2)$, and $[Y^{-\frac{1}{4},\frac{1}{2}+\vr}, Y^{\frac{3}{8}, \frac{3}{8}+\vr}]_{16\vr} = Y^{-\frac{1}{4}+10\vr, \frac{1}{2}-\vr}$, we have
\[
\| u \|_{L_{t,x}^8}
\lesssim \| u \|_{Y^{-\frac{1}{4}+10\vr, \frac{1}{2}-\vr}}.
\]

For the proof of \eqref{interp2}, we use the bilinear refinement of the Strichartz estimates:
For any $0<\vr_1, \vr_2 \ll 1$, we have
\[
\| P_{N_1}u P_{N_2}v \|_{L^2}
\lesssim N_1^{-2} \| P_{N_1} u \|_{Y^{0,\frac{1}{2}+\vr_1}} \| P_{N_2} v \|_{Y^{0,\frac{1}{2}+\vr_2}},
\]
which is a consequence of Lemma 3.2 in \cite{CLMW09}.
Because an interpolation yields that $Y^{-\frac{3}{4}, \frac{1}{2}+\vr_1} \hookrightarrow L_t^4(\R; L_x^{\infty}(\R))$ and $Y^{0, \frac{1}{4}+\frac{\vr_1}{2}} \hookrightarrow L_t^4 (\R; L_x^2(\R))$, we have
\begin{align*}
& \| P_{N_1}u P_{N_2}v \|_{L^2} \\
& \lesssim N_2^{-\frac{3}{4}} \min \left( \| P_{N_1} u \|_{Y^{0,\frac{1}{4}+\vr_1}} \| P_{N_2} v \|_{Y^{0,\frac{1}{2}+\vr_2}}, \| P_{N_1} u \|_{Y^{0,\frac{1}{4}+\vr_1}} \| P_{N_2} v \|_{Y^{0,\frac{1}{2}+\vr_2}} \right).
\end{align*}
By the interpolation $[Y^{0,\frac{1}{2}+\frac{\vr}{4}},Y^{0,\frac{1}{4}+\frac{\vr}{4}}]_{5\vr} = Y^{0,\frac{1}{2}-\vr}$, we obtain \eqref{interp2}.
\end{proof}

We observe the trilinear estimate in $Y^{s,b}$ spaces.

\begin{prop} \label{prop:trilinear}
Let $s_0, s_1, s_2, s_3 \in \R$.
Denote the decreasing rearrangement of $s_0,s_1,s_2,s_3$ by $s_0^{\ast}, s_1^{\ast}, s_2^{\ast}, s_3^{\ast}$ with $s_0^{\ast} \ge s_1^{\ast} \ge s_2^{\ast} \ge s_3^{\ast}$.
Assume that
\begin{align*}
& s_2^{\ast}+s_3^{\ast}>-4, \\
& s_1^{\ast}+s_2^{\ast}+s_3^{\ast}>-\frac{11}{4}, \\
& s_0+s_1+s_2+s_3 \ge -\frac{3}{2}.
\end{align*}
Then, there exists sufficiently small $\vr_0 = \vr_0 (s_0,s_1,s_2,s_3)>0$ such that for any $0 < \vr < \vr_0$, we have
\[
\| f_1f_2f_3 \|_{Y^{-s_0,-\frac{1}{2}+\vr}}
\lesssim \sum_{j=1}^3 \| f_j \|_{Y^{s_j, \frac{1}{2}-\vr}} \prod_{k \in \{ 1,2,3 \} \setminus \{ j \}} \| f_k \|_{Y^{s_k, \frac{1}{2}+\vr}}.
\]
\end{prop}

\begin{proof}
By duality argument and the Littlewood-Paley decompositions, it suffices to show that
\begin{align*}
& \sum_{N_0,N_1,N_2,N_3 \in 2^{\delta \N_0}} \left| \int_{\R^2} \ol{f_{N_0}} f_{N_1} f_{N_2} f_{N_3} dtdx \right| \\
& \lesssim \| f_0 \|_{Y^{s_0,\frac{1}{2}-\vr}} \sum_{j=1}^3 \| f_j \|_{Y^{s_j,\frac{1}{2}-\vr}} \prod_{k \in \{ 1,2,3 \} \setminus \{ j \}} \| f_k \|_{Y^{s_k,\frac{1}{2}+\vr}},
\end{align*}
where $f_{N_j} = P_{N_j} f_j$.
By symmetry, we may assume that
\[
N_1 \ge N_2 \ge N_3.
\]
We divide the proof into the three cases:
\begin{align*}
& \rm{(i)} \ N_0 \sim N_1 \sim N_3, \\
& \rm{(ii)} \ N_0 \sim N_1 \sim N_2 \gg N_3 \text{ or } N_1 \sim N_3 \gg N_0, \\
& \rm{(iii)} \ N_0 \sim N_1 \gg N_2 \ge N_3 \text{ or } N_1 \sim N_2 \gg N_0, N_3.
\end{align*}

Case (i): Proposition \ref{prop:bilinear} yields
\begin{align*}
& \sum_{\substack{N_0,N_1,N_2,N_3 \in 2^{\delta \N_0} \\ N_0 \sim N_1 \sim N_2 \sim N_3}} \left| \int_{\R^2} \ol{f_{N_0}} f_{N_1} f_{N_2} f_{N_3} dtdx \right| \\
& \le
\sum_{\substack{N_0,N_1,N_2,N_3 \in 2^{\delta \N_0} \\ N_0 \sim N_1 \sim N_2 \sim N_3}} \| \ol{f_{N_0}} f_{N_2} \|_{L^2} \| f_{N_1} f_{N_3} \|_{L^2} \\
& \lesssim \sum_{\substack{N_0,N_1,N_2,N_3 \in 2^{\delta \N_0} \\ N_0 \sim N_1 \sim N_2 \sim N_3}} N_0^{-\frac{3}{2}} \left( \prod_{j=0,1} \| f_{N_j} \|_{Y^{0,\frac{1}{2}-\vr}} \right) \left( \prod_{k=2,3} \| f_{N_k} \|_{Y^{0,\frac{1}{2}+\vr}} \right) \\
& \lesssim \| f_0 \|_{Y^{s_0,\frac{1}{2}-\vr}} \| f_1 \|_{Y^{s_1,\frac{1}{2}-\vr}} \| f_2 \|_{Y^{s_2,\frac{1}{2}+\vr}} \| f_3 \|_{Y^{s_3,\frac{1}{2}+\vr}},
\end{align*}
provided that $s_0+s_1+s_2+s_3 \ge -\frac{3}{2}$ and $\vr>0$ is sufficiently small.

Case (ii):
By Proposition \ref{prop:bilinear} and \eqref{interp2}, we have
\begin{align*}
\left| \int_{\R^2} \ol{f_{N_0}} f_{N_1} f_{N_2} f_{N_3} dtdx \right|
& \le
\| \ol{f_{N_0}} f_{N_3} \|_{L^2} \| f_{N_1} f_{N_2} \|_{L^2} \\
& \lesssim
N_1^{-\frac{11}{4}+10\vr} \left( \prod_{j=0,1} \| f_{N_j} \|_{Y^{0,\frac{1}{2}-\vr}} \right) \left( \prod_{k=2,3} \| f_{N_k} \|_{Y^{0,\frac{1}{2}+\vr}} \right).
\end{align*}
Because the condition (ii) implies that
\[
N_0^{-s_0} N_1^{-s_1} N_2^{-s_2} N_3^{-s_3}
\lesssim N_1^{-s_1^{\ast}-s_2^{\ast}-s_3^{\ast}} \min (N_0,N_3)^{-s_0^{\ast}},
\]
the summation under (ii) is bounded by
\[
\| f_0 \|_{Y^{s_0,\frac{1}{2}-\vr}} \| f_1 \|_{Y^{s_1,\frac{1}{2}-\vr}} \| f_2 \|_{Y^{s_2,\frac{1}{2}+\vr}} \| f_3 \|_{Y^{s_3,\frac{1}{2}+\vr}},
\]
provided that $s_1^{\ast}+s_2^{\ast}+s_3^{\ast} >-\frac{11}{4}$, $s_0+s_1+s_2+s_3>-\frac{11}{4}$, and $\vr>0$ is sufficiently small.

Case (iii):
By \eqref{interp2}, we have
\begin{align*}
\left| \int_{\R^2} \ol{f_{N_0}} f_{N_1} f_{N_2} f_{N_3} dtdx \right|
& \le \| \ol{f_{N_0}} f_{N_2} \|_{L^2} \| f_{N_1} f_{N_3} \|_{L^2} \\
& \lesssim
N_1^{-4+20\vr} \left( \prod_{j=0,1} \| f_{N_j} \|_{Y^{0,\frac{1}{2}-\vr}} \right) \left( \prod_{k=2,3} \| f_{N_k} \|_{Y^{0,\frac{1}{2}+\vr}} \right).
\end{align*}
Because the condition (iii) implies that
\begin{align*}
& N_0^{-s_0} N_1^{-s_1} N_2^{-s_2} N_3^{-s_3} \\
& \lesssim
\begin{cases}
N_1^{-s_2^{\ast}-s_3^{\ast}} N_2^{-s_1^{\ast}} N_3^{-s_0^{\ast}}, & \text{if } N_0 \sim N_1 \gg N_2 \ge N_3, \\ 
N_1^{-s_2^{\ast}-s_3^{\ast}} \max (N_0, N_3)^{-s_1^{\ast}} \min (N_0,N_3)^{-s_0^{\ast}}, & \text{if } N_1 \sim N_2 \gg N_0, N_3,
\end{cases}
\end{align*}
the summation under (iii) is bounded by
\[
\| f_0 \|_{Y^{s_0,\frac{1}{2}-\vr}} \| f_1 \|_{Y^{s_1,\frac{1}{2}-\vr}} \| f_2 \|_{Y^{s_2,\frac{1}{2}+\vr}} \| f_3 \|_{Y^{s_3,\frac{1}{2}+\vr}},
\]
provided that $s_2^{\ast}+s_3^{\ast} >-4$, $s_1^{\ast}+s_2^{\ast}+s_3^{\ast}>-4$, $s_0+s_1+s_2+s_3>-4$, and $\vr>0$ is sufficiently small.
\end{proof}

Moreover, we show the quinti-linear estimate in $Y^{s,b}$ spaces.

\begin{prop} \label{prop:quintilinear}
Let $s_0, s_1, s_2, s_3, s_4, s_5 \in \R$.
Denote the decreasing rearrangement of $s_0,s_1,s_2,s_3,s_4,s_5$ by $s_0^{\ast}, s_1^{\ast}, s_2^{\ast}, s_3^{\ast},s_4^{\ast},s_5^{\ast}$ with $s_0^{\ast} \ge s_1^{\ast} \ge s_2^{\ast} \ge s_3^{\ast} \ge s_4^{\ast} \ge s_5^{\ast}$.
Assume that
\begin{align*}
& s_4^{\ast}+s_5^{\ast}>-4, \\
& s_3^{\ast}+s_4^{\ast}+s_5^{\ast}>-\frac{7}{2}, \\
& s_2^{\ast}+s_3^{\ast}+s_4^{\ast}+s_5^{\ast}> -3, \\
& s_1^{\ast}+s_2^{\ast}+s_3^{\ast}+s_4^{\ast}+s_5^{\ast}>-3, \\
& s_0+s_1+s_2+s_3+s_4+s_5>-3.
\end{align*}
Then, there exists sufficiently small $\vr_0 = \vr_0 (s_0,s_1,s_2,s_3,s_4,s_5)>0$ such that for any $0 < \vr < \vr_0$, we have
\[
\| f_1f_2f_3f_4f_5 \|_{Y^{-s_0,-\frac{1}{2}+\vr}}
\lesssim \sum_{j=1}^5 \| f_j \|_{Y^{s_j, \frac{1}{2}-\vr}} \prod_{k \in \{ 1,2,3,4,5\} \setminus \{ j \}} \| f_k \|_{Y^{s_j, \frac{1}{2}+\vr}}.
\]
\end{prop}

\begin{proof}
By duality argument and the Littlewood-Paley decompositions, it suffices to show that
\begin{align*}
& \sum_{N_0,N_1,N_2,N_3,N_4,N_5 \in 2^{\delta \N_0}} \left| \int_{\R^2} \ol{f_{N_0}} \prod_{j=1}^5 f_{N_j} dtdx \right| \\
& \lesssim \| f_0 \|_{Y^{s_0, \frac{1}{2}-\vr}} \sum_{j=1}^5 \| f_j \|_{Y^{s_j, \frac{1}{2}-\vr}} \prod_{k \in \{ 1,2,3,4,5\} \setminus \{ j \}} \| f_k \|_{Y^{s_j, \frac{1}{2}+\vr}},
\end{align*}
where $f_{N_j} = P_{N_j} f_j$ for $0 \le j \le 5$.
By symmetry, we may assume that
\[
N_1 \ge N_2 \ge N_3 \ge N_4 \ge N_5.
\]
Then, it is reduced to showing that
\[
\left| \int_{\R^2} \ol{f_{N_0}} \prod_{j=1}^5 f_{N_j} dtdx \right|
\lesssim N_1^{-\vr} \| f_{N_0} \|_{Y^{s_0, \frac{1}{2}-\vr}} \sum_{j=1}^5 \| f_{N_j} \|_{Y^{s_j, \frac{1}{2}-\vr}} \prod_{k \neq j} \| f_{N_k} \|_{Y^{s_j, \frac{1}{2}+\vr}}.
\]
Here, the factor $N_1^{-\vr}$ is needed to sum up, because $N_1 \sim \max_{0 \le j \le 5} N_j$.
We divide the proof into the three cases:
\begin{align*}
& \rm{(i)} \ N_0 \sim N_1 \sim N_5, \\
& \rm{(ii)} \ N_0 \sim N_1 \sim N_4 \gg N_5 \text{ or }N_1 \sim N_5 \gg N_0, \\
& \rm{(iii)} \ N_0 \sim N_1 \gg N_4 \ge N_5 \text{ or } N_1 \sim N_2 \gg N_0, N_5.
\end{align*}

Case (i):
By \eqref{interp1}, we have
\begin{align*}
\left| \int_{\R^2} \ol{f_{N_0}} \prod_{j=1}^5 f_{N_j} dtdx \right|
& \le \prod_{j=0}^5 \| f_{N_j} \|_{L_{t,x}^6} \\
& \lesssim N_1^{-3-s_0-s_1-s_2-s_3-s_4-s_5+60\vr} \prod_{j=0}^5 \| f_{N_j} \|_{Y^{s_j, \frac{1}{2}-\vr}} \\
& \lesssim N_1^{-\vr} \prod_{j=0}^5 \| f_{N_j} \|_{Y^{s_j, \frac{1}{2}-\vr}},
\end{align*}
provided that $s_0+s_1+s_2+s_3+s_4+s_5 >-3$ and $\vr>0$ is sufficiently small.

Case (ii):
By \eqref{interp1} and \eqref{interp2}, we have
\begin{align*}
& \left| \int_{\R^2} \ol{f_{N_0}} \prod_{j=1}^5 f_{N_j} dtdx \right| \\
& \le \| \ol{f_{N_0}} f_{N_5} \|_{L^2} \prod_{j=1}^4 \| f_{N_j} \|_{L_{t,x}^8} \\
& \lesssim N_0^{-s_0} N_1^{-3-s_1-s_2-s_3-s_4+50\vr} N_5^{-s_5} \left( \prod_{j=0}^4 \| f_{N_j} \|_{Y^{s_j, \frac{1}{2}-\vr}} \right) \| f_{N_5} \|_{Y^{s_5, \frac{1}{2}+\vr}} \\
& \lesssim N_0^{-\vr} \left( \prod_{j=0}^4 \| f_{N_j} \|_{Y^{s_j, \frac{1}{2}-\vr}} \right) \| f_{N_5} \|_{Y^{s_5, \frac{1}{2}+\vr}},
\end{align*}
provided that
\[
s_1^{\ast}+s_2^{\ast}+s_3^{\ast}+s_4^{\ast}+s_5^{\ast}> -3, \quad s_0+s_1+s_2+s_3+s_4+s_5 >-3
\]
and $\vr>0$ is sufficiently small.

Case (iii):
There exists $\jmath \in \{ 2,4 \}$ such that $N_0 \not\sim N_{\jmath}$ and $N_1 \sim \max (N_0, N_{\jmath})$.
Set $\jmath' = \begin{cases} 4, & \text{if } \jmath=2, \\ 2, & \text{if } \jmath=4. \end{cases}$
By \eqref{interp2} and Bernstein's inequality, we have
\begin{align*}
\left| \int_{\R^2} \ol{f_{N_0}} \prod_{j=1}^5 f_{N_j} dtdx \right|
& \le \| \ol{f_{N_0}} f_{N_{\jmath}} \|_{L^2} \| f_{N_1} f_{N_5} \|_{L^2} \prod_{k= 3, \jmath'} \| f_{N_k} \|_{L_{t,x}^{\infty}} \\
& \lesssim N_0^{-s_0} N_1^{-4-s_1+20\vr} N_3^{\frac{1}{2}-s_3} N_5^{-s_5} N_{\jmath}^{-s_{\jmath}} N_{\jmath'}^{\frac{1}{2}-s_{\jmath'}} \\
& \quad \times 
\left( \prod_{l=0,1} \| f_{N_l} \|_{Y^{s_l,\frac{1}{2}-\vr}} \right)
\prod_{k=2}^5 \| f_{N_k} \|_{Y^{s_k, \frac{1}{2}+\vr}} \\
& \lesssim N_1^{-\vr} \| f_{N_0} \|_{Y^{s_0,\frac{1}{2}-\vr}} \| f_{N_1} \|_{Y^{s_1,\frac{1}{2}-\vr}} \prod_{k=2}^5 \| f_{N_k} \|_{Y^{s_k, \frac{1}{2}+\vr}},
\end{align*}
provided that
\begin{align*}
& s_4^{\ast}+s_5^{\ast}>-4, \quad
s_3^{\ast}+s_4^{\ast}+s_5^{\ast}>-\frac{7}{2}, \quad
s_2^{\ast}+s_3^{\ast}+s_4^{\ast}+s_5^{\ast}> -3, \\
& s_1^{\ast}+s_2^{\ast}+s_3^{\ast}+s_4^{\ast}+s_5^{\ast}>-3, \quad
 s_0+s_1+s_2+s_3+s_4+s_5>-3
\end{align*}
and $\vr>0$ is sufficiently small.
\end{proof}

We are now in the position to prove Proposition \ref{prop:WP}.

\begin{proof}[Proof of Proposition \ref{prop:WP}]
We define the operator $K_{u_0} (u)$ by
\[
K_{u_0} (u) = e^{\frac{1}{5} t\dx^5} u_0 + \int_0^t e^{\frac{1}{5} (t-t') \dx^5} \dx \left( \alpha (2 u^2 \dx^2 u + 3 u (\dx u)^2) + \beta u^5 \right) (t') dt'.
\]
Lemmas \ref{lem:Ylin}, \ref{lem:Ybb'} and Propositions \ref{prop:trilinear}, \ref{prop:quintilinear} yield
\begin{align*}
& \| K_{u_0} (u) \|_{Y^{s,\frac{1}{2}+\vr}_T} \\
& \le C_0 \| u (0) \|_{H^s} + C \Big( \| \sigma_{\le T} (t) u \dx u \dx^2u \|_{Y^{s,-\frac{1}{2}+\vr}} + \| \sigma_{\le T} (t) u^2 \dx^3 u\|_{Y^{s,-\frac{1}{2}+\vr}} \\
& \quad + \| \sigma_{\le T} (t) (\dx u)^3 \|_{Y^{s,-\frac{1}{2}+\vr}} + \| \sigma_{\le T} (t) u^4 \dx u \|_{Y^{s,-\frac{1}{2}+\vr}} \Big) \\
& \le C_0 \| u_0 \|_{H^s} + C_1 T^{\frac{\vr}{2}} \Big( \| u \|_{Y^{\frac{3}{4},\frac{1}{2}+\vr}}^2 \| u \|_{Y^{s,\frac{1}{2}+\vr}} + \| u \|_{Y^{\frac{3}{4}, \frac{1}{2}+\vr}}^4 \| u \|_{Y^{s,\frac{1}{2}+\vr}} \Big).
\end{align*}

A simple calculation yields that
\[
[\Lc,\Sc ]= 5 \Lc, \quad
[\Sc,\dx]=-\dx,
\]
which imply that for solutions $u$ to \eqref{5mKdV}
\begin{equation} \label{eq:lambda}
\begin{aligned}
\Lc \Lambda u
& = \dx^{-1} (S+5) \Lc u \\
& = \alpha \left( 4 u \dx^2u \dx \Lambda u + 2 u^2 \dx^3 \Lambda u + 3 (\dx u)^2 \dx \Lambda u + 6 u \dx u \dx^2 \Lambda u \right) \\
& \qquad + 5 \beta u^4 \dx \Lambda u.
\end{aligned}
\end{equation}
Accordingly, Lemmas \ref{lem:Ylin}, \ref{lem:Ybb'} and Propositions \ref{prop:trilinear}, \ref{prop:quintilinear} yield
\begin{align*}
& \| \Lambda K_{u_0} (u) \|_{Y^{0,\frac{1}{2}+\vr}_T} \\
& \le C_0 \| \Lambda u (0) \|_{L^2} + C \Big( \| \sigma_{\le T} (t) u \dx^2u \dx \Lambda u \|_{Y^{0,-\frac{1}{2}+\vr}} + \| \sigma_{\le T} (t) u^2 \dx^3 \Lambda u \|_{Y^{0,-\frac{1}{2}+\vr}} \\
& \quad + \| \sigma_{\le T} (t) (\dx u)^2 \dx \Lambda u \|_{Y^{0,-\frac{1}{2}+\vr}} + \| \sigma_{\le T} (t) u \dx u \dx^2 \Lambda u \|_{Y^{0,-\frac{1}{2}+\vr}} \\
& \quad + \| \sigma_{\le T} (t) u^4 \dx \Lambda u \|_{Y^{0,-\frac{1}{2}+\vr}} \Big) \\
& \le C_0 \| x u_0 \|_{L^2} + C_1 T^{\frac{\vr}{2}} \Big( \| u \|_{Y^{\frac{3}{4},\frac{1}{2}+\vr}}^2 \| \Lambda u \|_{Y^{0,\frac{1}{2}+\vr}} + \| u \|_{Y^{\frac{3}{4}, \frac{1}{2}+\vr}}^4 \| \Lambda u \|_{Y^{0,\frac{1}{2}+\vr}} \Big).
\end{align*}
Hence, we have
\[
\| K_{u_0} (u) \|_{Z^{s,\frac{1}{2}+\vr}_T}
\le C_0 \| u_0\|_{H^{s,1}} + C_1T^{\frac{\vr}{2}} \left( \| u \|_{Z^{\frac{3}{4}, \frac{1}{2}+\vr}_T}^2 + \| u \|_{Z^{\frac{3}{4}, \frac{1}{2}+\vr}_T}^4 \right) \| u \|_{Z^{s, \frac{1}{2}+\vr}_T}.
\]
Let $u_1$ and $u_2$ satisfy \eqref{5mKdV} with the same initial data.
Then, a similar mummer yields
\begin{align*}
& \| K_{u_0} (u_1) - K_{u_0} (u_2) \|_{Z^{s,\frac{1}{2}+\vr}_T} \\
& \le C_1 T^{\frac{\vr}{2}} \left( \| u_1 \|_{Z^{\frac{3}{4},\frac{1}{2}+\vr}_T}^2 + \| u_2 \|_{Z^{\frac{3}{4},\frac{1}{2}+\vr}_T}^2 + \| u_1 \|_{Z^{\frac{3}{4}, \frac{1}{2}+\vr}_T}^4 + \| u_2 \|_{Z^{\frac{3}{4}, \frac{1}{2}+\vr}_T}^4 \right) \\
& \quad \times \| u_1-u_2 \|_{Z^{s,\frac{1}{2}+\vr}_T}.
\end{align*}
Hence, taking $T \in (0,1)$ with
\[
100 C_1 T^{\frac{\vr}{2}} \left\{ (C_0 \| u_0 \|_{H^{\frac{3}{4},1}})^2 + (C_0 \| u_0 \|_{H^{\frac{3}{4},1}})^4 \right\} \le 1,
\]
we obtain that the mapping $K_{u_0}$ is a contraction on the ball $B(2C_0 \| u_0 \|_{H^{s,1}}) := \{ u \in Z^{s,\frac{1}{2}+\vr}_T \colon \| u \|_{Z^{s,\frac{1}{2}+\vr}_T} \le 2C_0 \| u_0 \|_{H^{s,1}} \}$.
Accordingly, there exists a unique $u \in B (2C_0 \| u_0 \|_{H^{s,1}})$ with $u = K_{u_0} (u)$.
Because the remaining properties follow from the standard argument, we omit details here.
\end{proof}

Corollary \ref{cor:WP} follows from the same manner.

\section{Energy estimates} \label{S:energy}

We show energy estimates for solutions $u$ to \eqref{5mKdV}.
For the estimate of $\| u \|_{H^2}$, we need to add some correction term.

\begin{lem} \label{lem:energy}
Let $u$ be a solution to \eqref{5mKdV} in a time interval $[0,T]$ satisfying
\[
\| u_0 \|_{H^{2,1}} \le \eps \ll 1
\]
and \eqref{est:u_infty0}.
Then,
\[
\| u(t) \|_{X^2} \le C_1 \eps \lr{t}^{C_2 \eps},
\]
where $C_1$ and $C_2$ are constants depending only on $|\alpha|$ and $|\beta|$.
In particular, $C_1$ and $C_2$ do not depend on $D$, $T$, and $\eps$.
\end{lem}

\begin{proof}
Because Proposition \ref{prop:WP} yields
\[
\sup_{0 \le t \le 1} \| u(t) \|_{X^2} \lesssim \eps,
\]
we consider the case $t \ge 1$.
By \eqref{est:u_infty0} and $D\eps \le \eps^{\frac{1}{2}}$, we have
\begin{equation} \label{Linfest}
\| u(t) \dx^3u(t) \|_{L^{\infty}}, \| \dx u(t) \dx^2 u(t) \|_{L^{\infty}}, \| u(t)^3 \dx u(t) \|_{L^{\infty}}
\le \eps t^{-1}.
\end{equation}
Integration by parts and \eqref{Linfest} yield
\begin{align*}
\frac{1}{2} \dt \| u (t) \|_{L^2}^2
&= -\int_{\R} \dx u \cdot (\alpha (2u^2 \dx^2 u + 3 u (\dx u)^2) + \beta u^5) dx
= \alpha \int_{\R} u^2 \dx u \dx^2u dx \\
&\le |\alpha| \| \dx u(t) \dx^2 u(t) \|_{L^{\infty}} \| u(t) \|_{L^2}^2
\le |\alpha| \eps t^{-1} \| u(t) \|_{L^2}^2.
\end{align*}
A similar calculation shows
\begin{align*}
\frac{1}{2} \dt \| \dx u (t) \|_{L^2}^2
&= -\int_{\R} \dx^2 u \cdot \dx (\alpha (2u^2 \dx^2 u + 3 u (\dx u)^2) + \beta u^5) dx \\
&= 4\alpha \int_{\R} u \dx^3u (\dx u)^2 dx +10 \beta \int_{\R} u^3 (\dx u)^3 dx \\
&\le \left( 4 |\alpha| \| u(t) \dx^3 u (t) \|_{L^{\infty}} + 10 |\beta| \| u (t)^3 \dx u(t) \|_{L^{\infty}} \right) \| \dx u (t) \|_{L^2}^2 \\
& \le \left( 4 |\alpha| + 10 |\beta| \right) \eps t^{-1} \| u(t) \|_{H^1}^2.
\end{align*}
Moreover, we have
\begin{align*}
\frac{1}{2} \dt \| \dx^2 u (t) \|_{L^2}^2
&= -\int_{\R} \dx^3 u \cdot \dx^2 (\alpha (2u^2 \dx^2 u + 3 u (\dx u)^2) + \beta u^5) dx \\
&= -12 \alpha \int_{\R} u \dx u (\dx ^3 u)^2 dx + \frac{67}{3} \alpha \int_{\R} \dx u (\dx^2 u)^3 dx \\
& \quad -30 \beta \int_{\R} u (\dx u)^5 dx + 50 \beta \int_{\R} u^3 \dx u (\dx^2u)^2 dx \\
& = -12 \alpha \int_{\R} u \dx u (\dx ^3 u)^2 dx + O \left( \eps t^{-1} \| u(t) \|_{H^2}^2 \right).
\end{align*}
Because integration by parts and \eqref{Linfest} imply
\begin{align*}
& \frac{1}{2} \dt \int_{\R} u(t)^2 (\dx u(t))^2 dx \\
&= -\frac{1}{5} \int_{\R} \left( u (\dx u)^2 + u^2 \dx^2 u \right) \dx^5 u dx - \int_{\R} \left( u (\dx u)^2 + u^2 \dx^2 u \right) \Lc u dx \\
&= - \int_{\R} u \dx u (\dx^3 u)^2 dx + O \left( \eps t^{-1} \| u(t) \|_{H^2}^2 \right),
\end{align*}
we obtain
\[
\frac{1}{2} \dt \left( \| \dx^2 u (t) \|_{L^2}^2 -12 \alpha \int_{\R} u(t)^2 (\dx u(t))^2 dx \right)
\le C \eps t^{-1} \| u(t) \|_{H^2}^2.
\]
By \eqref{eq:lambda}, integration by parts, and \eqref{Linfest}, we have
\begin{equation} \label{energylam}
\begin{aligned}
& \frac{1}{2} \dt \| \Lambda u(t) \|_{L^2}^2 \\
& = -\alpha \int_{\R} (2 \dx u \dx^2 u + u \dx^3 u) (\Lambda u)^2 dx -10 \beta \int_{\R} u^3 \dx u (\Lambda u)^2 dx \\
& \le C \left( \| \dx u(t) \dx^2 u(t) \|_{L^{\infty}} + \| u(t) \dx^3 u(t) \|_{L^{\infty}} + \| u(t)^3 \dx u(t) \|_{L^{\infty}} \right) \| \Lambda u(t) \|_{L^2}^2 \\
& \le C \eps t^{-1} \| \Lambda u(t) \|_{L^2}^2
\end{aligned}
\end{equation}
Because
\[
\left| \int_{\R} u(t)^2 (\dx u(t))^2 dx \right|
\le \| u(t) \|_{L^{\infty}}^2 \| \dx u(t) \|_{L^2}^2
\le \eps t^{-\frac{2}{5}} \| u(t) \|_{H^1}^2,
\]
Gronwall's inequality with above estimates implies
\[
\| u(t) \|_{X^2}
\le 10 \| u(1) \|_{X^2} t^{C \eps}
\lesssim \eps t^{C \eps}.
\]
\end{proof}

For $\alpha=0$, we use the Kato-Ponce commutator estimate (see \cite{KatPon88}, \cite{KPV93}).

\begin{lem} \label{lem:KP}
For $0<s<1$, we have
\[
\| |\dx|^s (fg) - f |\dx|^s g \|_{L^2}
\lesssim \| |\dx|^s f \|_{L^2} \| g \|_{L^{\infty}}.
\]
\end{lem}

\begin{lem} \label{lem:energy2}
Let $0<s<1$ and let $u$ be a solution to \eqref{5mKdV5} in a time interval $[0,T]$ satisfying
\[
\| u_0 \|_{H^{s,1}} \le \eps \ll 1
\]
and \eqref{est:u_infty0}.
Then,
\[
\| u(t) \|_{X^s} \le C_1 \eps \lr{t}^{C_2 \eps},
\]
where $C_1$ and $C_2$ are constants depending only on $|\alpha|$, $|\beta|$, and $s$.
\end{lem}

\begin{proof}
As in the proof of Lemma \ref{lem:energy}, we have
\[
\| u(t) \|_{L^2} = \| u_0 \|_{L^2} \le \eps, \quad \| \Lambda u(t) \|_{L^2} \le C_1 \eps \lr{t}^{C_2 \eps}.
\]
It remains to estimate $\| u(t) \|_{\dot{H}^s}$.
Lemma \ref{lem:KP} and \eqref{est:u_infty0} yield
\begin{align*}
\frac{1}{2} \dt \| u(t) \|_{\dot{H}^s}^2
&= 5\beta \int_{\R} |\dx|^s u \cdot |\dx|^s (u^4 \dx u) dx \\
&= 5\beta \int_{\R} |\dx|^s u \cdot u^4 |\dx|^s \dx u dx + O \left( \| u (t) \|_{\dot{H}^s}^2 \| u(t) \|_{L^{\infty}}^3 \| \dx u (t) \|_{L^{\infty}} \right) \\
&= -10 \beta \int_{\R} u^3 \dx u (|\dx|^s u)^2 dx + O \left( \eps t^{-1} \| u(t) \|_{\dot{H}^s}^2 \right) \\
& \lesssim \eps t^{-1} \| u(t) \|_{\dot{H}^s}^2.
\end{align*}
By using Gronwall's inequality, we obtain the desired bound.
\end{proof}

We define the auxiliary space
\[
\| u(t) \|_{\wt{X}} := \| \J u(t) \|_{L^2} + t^{\frac{1}{5}} \| \lr{t^{\frac{1}{5}} \dx}^{-1} u(t) \|_{L^2}.
\]

\begin{lem} \label{lem:Xtilde}
Let $u$ be a solution to \eqref{5mKdV} satisfying $\| u_0 \|_{H^{0,1}} \le \eps \ll 1$ and \eqref{est:u_infty0}.
Then, for $t \ge 1$, we have
\[
\| u(t) \|_{\wt{X}} \lesssim \eps t^{\frac{1}{10}},
\]
where the implicit constant depends only on $|\alpha|$ and $|\beta|$.
\end{lem}

\begin{proof}
We note that \eqref{eq:deflambda} implies
\[
\J u
= \Lambda u - 5t \dx^{-1} \Lc u
= \Lambda u - 5t \left( \alpha (2 u^2 \dx^2 u + 3 u (\dx u)^2) + \beta u^5 \right).
\]
Because \eqref{est:u_infty0} yields
\[
|u (t,x)^2 \dx^2 u (t,x)|, |u(t,x) (\dx u (t,x))^2|, |u(t,x)|^5
\le \eps t^{-1} \lr{t^{-\frac{1}{5}} x}^{-\frac{5}{8}},
\]
we have
\begin{align*}
& \|u^2 \dx^2 u\|_{L^2} +  \|u (\dx u)^2\|_{L^2} + \| u^5 \|_{L^2} \\
& \lesssim \eps t^{-1} \left( \int_{|x| \le t^{\frac{1}{5}}} dx + \int_{|x| \ge t^{\frac{1}{5}}} \left( t^{-\frac{1}{5}} |x| \right)^{-\frac{5}{4}} dx \right)^{\frac{1}{2}} \\
& \lesssim \eps t^{-1+\frac{1}{10}}.
\end{align*}
We apply Gronwall's inequality with \eqref{energylam} to obtain $\| \Lambda u(t) \|_{L^2} \le C_1 \eps t^{C_2 \eps}$.
We therefore have
\begin{equation} \label{est:JuL2}
\| \J u(t) \|_{L^2}
\lesssim \| \Lambda u(t) \|_{L^2} + \|u^2 \dx^2 u\|_{L^2} +  \|u (\dx u)^2\|_{L^2} + \| u^5 \|_{L^2}
\lesssim \eps t^{\frac{1}{10}}.
\end{equation}

We use a self-similar change of variables by defining
\begin{equation} \label{selfsimilar}
U(t,y) := t^{\frac{1}{5}} u(t,t^{\frac{1}{5}}y).
\end{equation}
From $\Lambda u = 5t \dx^{-1} \dt u + x u$, a direct calculation shows
\begin{equation} \label{eq:U}
\partial_t U(t,y)
= \frac{1}{5} t^{-1} \dy \left( (\Lambda u) (t,t^{\frac{1}{5}}y) \right)
\end{equation}
Hence, we have
\[
\partial_t \| \lr{\dy}^{-1} U(t) \|_{L^2_y}
\lesssim t^{-\frac{11}{10}} \| \Lambda u (t) \|_{L^2_x}
\lesssim \eps t^{-\frac{11}{10}+C_2 \eps} .
\]
Integrating this with respect to $t$, we have
\[
\| \lr{\dy}^{-1} U (t) \|_{L^2_y} \lesssim \eps .
\]
From $\| \lr{\dy}^{-1} U (t) \|_{L^2_y} = t^{\frac{1}{10}} \| \lr{t^{\frac{1}{5}} \dx}^{-1} u (t) \|_{L^2_x}$, we obtain the desired bound.
\end{proof}

\begin{rmk} \label{rmk:reqH2}
The estimate $\| u(t) \|_{\wt{X}} \lesssim \eps$ for $0<t<1$ holds true if $u$ has regularity.
Indeed, as in \eqref{est:JuL2},
\begin{align*}
\| \J u(t) \|_{L^2}
& \lesssim \| \Lambda u(t) \|_{L^2} + \| u(t) \|_{L^{\infty}}^2 \| u(t) \|_{H^2} + \| u(t) \|_{L^6} \| \dx u(t) \|_{L^6}^2 + \| u(t) \|_{L^{10}}^5 \\
& \lesssim \| \Lambda u(t) \|_{L^2} + \| u(t) \|_{H^2}^3 + \| u(t) \|_{H^1}^5.
\end{align*}
Accordingly, by Proposition \ref{prop:WP} and taking $\eps>0$ sufficiently small, we have
\[
\sup_{0 \le t \le 1} \| u(t) \|_{\wt{X}}
\lesssim \sup_{0 \le t \le 1} \left( \| u(t) \|_{X^2} + \| u(t) \|_{X^2}^5 \right)
\lesssim \eps.
\]
For $\alpha=0$, because the cubic part vanishes, by $H^{\frac{2}{5}} (\R) \hookrightarrow L^{10}(\R)$, we have
\[
\sup_{0 \le t \le 1} \| u(t) \|_{\wt{X}}
\lesssim \sup_{0 \le t \le 1} \left( \| u(t) \|_{X^{\frac{2}{5}}} + \| u(t) \|_{X^{\frac{2}{5}}}^5 \right)
\lesssim \eps.
\]
\end{rmk}

\section{Decay estimates} \label{S:KS_type}

We decompose $u$ into positive and negative frequencies:
\[
u = u^+ + u^-, \quad u^{\pm} := P^{\pm} u.
\]
Because $u$ is real valued, $u^+ = \overline{u^-}$ and $u= 2 \Re u^+$.
We write $u_N := P_N u$ and $u_N^+ := P_N^+ u$.
From $\J u_N = P_N (\J u) +i N^{-1}\mathcal{F}^{-1} [ \sigma'(\frac{\xi}{N}) \wh{u}]$, we have
\[
\| u (t) \| _{\wt{X}} \sim \Bigg( \|u_{\le t^{-\frac{1}{5}}} (t) \| _{\wt{X}}^2 + \sum_{\substack{N \in 2^{\delta \mathbb{Z}} \\ N \ge t^{-\frac{1}{5}}}} \|u_N (t) \| _{\wt{X}}^2 \Bigg)^{\frac{1}{2}}.
\]

For $t \ge 1$, we further decompose $u^+$ into the hyperbolic and elliptic parts
\[
\uhp = \sum _{\substack{N \in 2^{\delta \mathbb{Z}} \\ N \ge t^{-\frac{1}{5}}}} \uhp_{N}, \quad
\uep = u^+ - \uhp,
\]
where, for $N \ge t^{-\frac{1}{5}}$, we define
\[
\uhp_{N} := \sigma_N^{\rm{hyp}} u^{+}_{N}, \quad
\uep_{N} := u^+_{N} - \uhp_{N}.
\]
Here, $\sigma_N^{\rm{hyp}} (t,x) := \sigma_{\frac{1}{3} t N^4 \le \cdot \le 3 t N^4} (x) \bm{1}_{\R_-} (x)$.

We note that $\uhp$ is supported in $\{ x \in \R_- \colon t^{-\frac{1}{5}} |x| \ge \frac{1}{3 \cdot 2^{\delta}} \}$.
For $(t,x) \in \R^2$ with $t^{-\frac{1}{5}} |x| \ge \frac{1}{3 \cdot 2^{\delta}}$, the number of scaled dyadic numbers $N \in 2^{\delta \mathbb{Z}}$ satisfying $\frac{1}{3 \cdot 2^{\delta}} tN^4 \le |x| \le 3 \cdot 2^{\delta} t N^4$ is  less than $\frac{1}{\delta}$.
Hence, $\uhp (t,x)$ is a finite sum of $\uhp_N(t,x)$'s.

The functions $\uh_N$ and $\ue_N$ are essentially frequency localized near $N$, which is a consequence of the following.

\begin{lem} \label{lem:freq_spat_loc}
For $2 \le p \le \infty$, any $a ,\, b ,\, c \in \R$ with $a \ge 0$ and $a+c \ge 0$, and any $R>0$, we have
\[
\| (1- P_{\frac{N}{2^{\delta}} \le \cdot \le 2^{\delta} N}) | \dx| ^{a} ( |x|^{b} \sigma_R P_N f) \|_{L^p}
\lesssim_{a,b,c} N^{-c+\frac{1}{2}-\frac{1}{p}} R^{-a+b-c}  \| P_N f \|_{L^2} . 
\]
Moreover, we may replace $\sigma_R$ on the left hand side by $\sigma_{>R}$ if $a+c>b+1$ and $\sigma_{<R}$ if $a+c \ge 0$ and $b=0$.

In addition, for any $0<r<R$, we have
\begin{align*}
& \| (1- P_{\frac{N}{2^{\delta}} \le \cdot \le 2^{\delta} N}) | \dx| ^{a} ( |x|^{b} \sigma_{r< \cdot <R} P_N f) \|_{L^2} \\
& \quad \lesssim_{a,b,c} N^{-c} R^{-a+b-c} \left( \frac{R}{r} \right)^{a+|b|+c+2} \| P_N f \|_{L^2} . 
\end{align*}
\end{lem}

\begin{proof}
The first inequality is proved in Lemma 3.1 in \cite{O}.

Because integration by parts yields
\[
| |\mathcal{F} [ |x|^b \sigma_{\frac{r}{R}< \cdot <1}](\xi) | \lesssim |\xi|^{-k} \left( \frac{R}{r} \right)^{k+|b|}
\]
for any $k \in \mathbb{N}_0$, we obtain
\[
\| \mathcal{F} [ |\dx|^{a+c} (|x|^b \sigma_{\frac{r}{R}< \cdot <1})] \|_{L^1}
= \| |\xi|^{a+c} \mathcal{F} [ |x|^b \sigma_{\frac{r}{R}< \cdot <1}] \|_{L^1}
\lesssim \left( \frac{R}{r} \right)^{a+|b|+c+2}
\]
provided that $a \ge 0$ and $a+c \ge 0$.
Hence, the second inequality follows from the same argument as in Lemma 3.1 in \cite{O}.
\end{proof}

Lemma \ref{lem:freq_spat_loc} yields that for any $a \ge 0$, $b \in \R$, and $c \ge 0$,
\begin{align}
& \label{est:uhperr}
\| (1-\PN) |\partial_x|^{a} (|x|^{b} \uhp_N) \|_{L^2}
\lesssim_{a,b,c} t^{-\frac{a-b}{5}} (t^{\frac{1}{5}} N)^{-c} \| u_N \|_{L^2},
\\
& \label{est:ueperr}
\| (1-\PN) |\partial_x|^{a} \uep_N \|_{L^2}
\lesssim_{a,c} t^{-\frac{a}{5}} (t^{\frac{1}{5}} N)^{-c} \| u_N \|_{L^2},
\\
& \label{est:ueperr2}
\| (1-\PN) |\partial_x|^{a} (|x|^b \sigma_{>t^{\frac{1}{5}}} (x) \uep_N) \|_{L^2}
\lesssim_{a,b,c} t^{-\frac{a-b}{5}} (t^{\frac{1}{5}} N)^{-c} \| u_N \|_{L^2}.
\end{align}

Factorizing the symbol $x+t\xi^4$ of $\J$, we define
\[
\J_+ := |x|^{\frac{1}{4}} +i t^{\frac{1}{4}} \dx, \quad
\J_- := \sum_{k=0}^3 |x|^{\frac{3-k}{4}} \left( -i t^{\frac{1}{4}} \dx \right)^k .
\]
These operators are useful in our analysis.

First, we show the following frequency localized estimates.

\begin{lem} \label{lem:he_freq_est}
For $t \ge 1$ and $N \ge t^{-\frac{1}{5}}$, we have
\begin{align}
& \left\| (|x|^{\frac{3}{4}} + t^{\frac{3}{4}} N^{3}) \J_+ \uhp_{N} (t) \right\|_{L^2} \lesssim \| u_N (t) \|_{\wt{X}}, \label{uhpNl2} \\
& \left\|(|x|+ t N^{4}) \uep_{N} (t) \right\|_{L^2} \lesssim \| u_N (t) \|_{\wt{X}}. \label{uepNl2}
\end{align}
\end{lem}

\begin{proof}
Set $f := \J_+ \uhp_N$.
Because \eqref{est:uhperr} yields
\begin{align*}
\left| t^{\frac{2}{5}(k-b)+\frac{1}{2}} \int_{\R_-} \xi |\mathcal{F} [ |x|^b \dx^k f ]|^2 d\xi \right|
& \le t^{\frac{2}{5}(k-b)+\frac{1}{2}} \| P^- |\dx|^{\frac{1}{2}} ( |x|^b \dx^k \J_+ \uhp_N) \|_{L^2}^2 \\
& \lesssim N^{-2} \| u_N \|_{L^2}^2,
\end{align*}
we have
\begin{equation} \label{eq:uhpJ-}
\begin{aligned}
\| \J_-f \|_{L^2}^2
=& \sum_{k=0}^{3} \left\| t^{\frac{k}{4}} |x|^{\frac{3-k}{4}} \dx^k f \right\|_{L^2}^2
-2 t^{\frac{1}{4}} \Im \int_{\R} |x|^{\frac{5}{4}} f(x) \ol{\dx f(x)} dx \\
& -2 t^{\frac{1}{2}} \Re \int_{\R} |x| f(x) \ol{\dx^2 f(x)} dx
+2 t^{\frac{3}{4}} \Im \int_{\R} |x|^{\frac{3}{4}} f(x) \ol{\dx^3 f(x)} dx \\
& -2 t^{\frac{3}{4}} \Im \int_{\R} |x|^{\frac{3}{4}} \dx f \ol{\dx^2 f(x)} dx
-2 t \Re \int_{\R} |x|^{\frac{1}{2}} \dx f(x) \ol{\dx^3 f(x)} dx \\
& -2 t^{\frac{5}{4}} \Im \int_{\R} |x|^{\frac{1}{4}} \dx^2 f(x) \ol{\dx^3 f(x)} dx \\
=& \sum_{k=0}^{3} \left\| t^{\frac{k}{4}} |x|^{\frac{3-k}{4}} \dx^k f \right\|_{L^2}^2
+2 t^{\frac{1}{4}} \int_{\R} \xi |\mathcal{F} [ |\cdot|^{\frac{5}{8}} f] (\xi)|^2 d\xi \\
& +2 t^{\frac{1}{2}} \int_{\R} |x| |\dx f(x)|^2 dx
+\frac{3}{8} t^{\frac{3}{4}} \int_{\R} \xi |\mathcal{F} [ |\cdot|^{-\frac{5}{8}}f] (\xi)|^2 d\xi \\
& +4t^{\frac{3}{4}} \int_{\R} \xi |\mathcal{F} [|\cdot|^{\frac{3}{8}} \dx f] (\xi)|^2 d\xi
+2t \int_{\R} |x|^{\frac{1}{2}} |\dx^2 f(x)|^2 dx \\
& +\frac{1}{4} t \int_{\R} |x|^{-\frac{3}{2}} |\dx f(x)|^2 dx
+2 t^{\frac{5}{4}} \int_{\R} \xi |\mathcal{F} [ |\cdot|^{\frac{1}{8}} \dx^2 f] (\xi)|^2 d\xi \\
\ge & \sum_{k=0}^{3} \left\| t^{\frac{k}{4}} |x|^{\frac{3-k}{4}} \dx^k f \right\|_{L^2}^2
+2 t^{\frac{1}{4}} \int_{\R_-} \xi |\mathcal{F} [ |\cdot|^{\frac{5}{8}} f] (\xi)|^2 d\xi \\
& +\frac{3}{8} t^{\frac{3}{4}} \int_{\R_-} \xi |\mathcal{F} [ |\cdot|^{-\frac{5}{8}}f] (\xi)|^2 d\xi
+4t^{\frac{3}{4}} \int_{\R_-} \xi |\mathcal{F} [|\cdot|^{\frac{3}{8}} \dx f] (\xi)|^2 d\xi \\
& +2 t^{\frac{5}{4}} \int_{\R_-} \xi |\mathcal{F} [ |\cdot|^{\frac{1}{8}} \dx^2 f] (\xi)|^2 d\xi \\
\gtrsim & \sum_{k=0}^{3} \left\| t^{\frac{k}{4}} |x|^{\frac{3-k}{4}} \dx^k f \right\|_{L^2}^2
- N^{-2} \| u_N \|_{L^2}^2.
\end{aligned}
\end{equation}
Because
\begin{align*}
& \J_- \J _+ \uhp_N \\
& = -\J \uhp_N + \frac{i}{4} t^{\frac{1}{4}} |x|^{-\frac{1}{4}} \uhp_N + \frac{3}{16} t^{\frac{1}{2}} |x|^{-\frac{3}{2}} \uhp_N + \frac{1}{2} t^{\frac{1}{2}} |x|^{-\frac{1}{2}} \dx \uhp_N \\
& \quad  - \frac{21}{63}i t^{\frac{3}{4}} |x|^{-\frac{11}{4}} \uhp_N - \frac{9}{16}i t^{\frac{3}{4}} |x|^{-\frac{7}{4}} \dx \uhp_N - \frac{3}{4} i t^{\frac{3}{4}} |x|^{-\frac{3}{4}} \dx^2 \uhp_N
\end{align*}
and
\begin{align*}
\J \uhp_N = & \sigma_N^{\rm{hyp}} \J u^+_N + t \Big( \dx^4\sigma_N^{\rm{hyp}} \cdot u^+_N + 4\dx^3 \sigma_N^{\rm{hyp}} \cdot \dx u^+_N \\
& \hspace*{70pt} + 6 \dx^2 \sigma_N^{\rm{hyp}} \cdot \dx^2 u^+_N+ 4 \dx \sigma_N^{\rm{hyp}} \cdot \dx^3 u^+_N \Big),
\end{align*}
by \eqref{est:uhperr} and $tN^5 \ge 1$, we have
\[
\| \J_- \J_+ \uhp_N \|_{L^2}
\lesssim \| \J u_N \|_{L^2} + N^{-1} \| u_N \|_{L^2}.
\]
Combining this with \eqref{eq:uhpJ-}, we obtain
\[
\| (|x|^{\frac{3}{4}} + t^{\frac{3}{4}} \dx^3) \J_+ \uhp_N \| _{L^2} \lesssim \| u_N (t) \|_{\wt{X}}.
\]

For the elliptic bound, we decompose $\uep_N$ into three parts
\[
\uep_N = \sigma _{\le \frac{1}{3} tN^{4}} \uep_{N} + \sigma _{\frac{1}{3} tN^{4} < \cdot < 3 tN^{4}} \uep_N + \sigma _{\ge 3 tN^{4}} \uep_N.
\]
We observe that the equation
\begin{equation} \label{eq:ueob}
\left\| x f \right\|_{L^2}^2 + \| t\dx^{4} f \|_{L^2}^2 = \| \J f \|_{L^2}^2 - 2 \int_{\R} tx |\dx^2 f(x)|^2 dx
\end{equation}
holds for any smooth function $f$.

By $(a+b)^2 \le (1+\delta) a^2+ (1+\delta^{-1}) b^2$ and \eqref{est:ueperr2}, we have
\begin{align*}
& \left| \int_{\R} tx | \dx^2 ( \sigma _{\ge 3 tN^{4}} \uep_{N} ) (t,x)|^2 dx \right| \\
& \le \frac{2^{\delta}}{3} N^{-4} \left\| x \dx^2 ( \sigma _{\ge 3 tN^{4}} \uep_{N} (t) ) \right\|_{L^2}^2 \\
& \le \frac{(1+\delta)2^{\delta}}{3} N^{-4} \left\| \dx^2 ( x \sigma _{\ge 3 tN^{4}} \uep_{N} (t) ) \right\|_{L^2}^2 + C N^{-4} \left\| \dx ( \sigma _{\ge 3 tN^{4}} \uep_{N} (t) ) \right\|_{L^2}^2 \\
& \le \frac{(1+\delta)^2 2^{9\delta}}{3} \left\| \PN \left( x \sigma _{\ge 3 tN^{4}} \uep_{N} (t) \right) \right\|_{L^2}^2 \\
& \quad + C N^{-4} \left\| (1-\PN) \dx^2 ( x \sigma _{\ge 3 tN^{4}} \uep_{N} (t) ) \right\|_{L^2}^2 \\
& \quad + C N^{-2} \| u_N (t) \|_{L^2}^2 \\
& \le \frac{(1+\delta)^2 2^{9\delta}}{3} \left\| x \sigma _{\ge 3 tN^{4}} \uep_{N} (t) \right\|_{L^2}^2 + C N^{-2} \| u_{N} (t) \|_{L^2}^2 .
\end{align*}
Taking $f= \sigma _{\ge 3 tN^{4}} \uep_{N}$ in \eqref{eq:ueob}, by $2 \frac{(1+\delta)^2 2^{9\delta}}{3} <1$, we have
\[
\left\| x \sigma _{\ge 3 tN^{4}} \uep_{N} (t) \right\|_{L^2}
\lesssim \| u_N (t) \|_{\wt{X}}.
\]

By $(a+b)^2 \le (1+\delta) a^2+ (1+\delta^{-1}) b^2$ and \eqref{est:ueperr}, we have
\begin{align*}
& \left| \int_{\R} tx | \dx^2 ( \sigma _{\le \frac{1}{3} tN^{4}} \uep_{N} ) (t,x)|^2 dx \right| \\
& \le \frac{2^{\delta}}{3} N^{4} \left\| t \dx^2 ( \sigma _{\le \frac{1}{3} tN^{4}} \uep_{N} (t) ) \right\|_{L^2}^2 \\
& \le \frac{(1+\delta) 2^{9\delta}}{3} \left\| \PN t \dx^{4} ( \sigma _{\le \frac{1}{3} tN^{4}} \uep_{N} (t) ) \right\|_{L^2}^2  \\
& \quad + C N^{4} \left\| (1-\PN) t \dx^2 ( \sigma _{\le \frac{1}{3} tN^{4}} \uep_{N} (t) ) \right\|_{L^2}^2 \\
& \le \frac{(1+\delta) 2^{9\delta}}{3} \left\| t \dx^{4} (\sigma _{<\frac{1}{3} tN^{4}} \uep_{N} (t)) \right\|_{L^2}^2 + C N^{-2} \| u_N (t) \|_{L^2}^2 .
\end{align*}
Taking $f= \sigma _{\le \frac{1}{3} tN^{4}} \uep_{N}$ in \eqref{eq:ueob}, by $2 \frac{(1+\delta) 2^{9\delta}}{3} <1$, we have
\[
\| t \dx^{4} (\sigma _{\le \frac{1}{3} tN^{4}} \uep_{N} (t)) \|_{L^2}
\lesssim \| u_N (t) \|_{\wt{X}}.
\]

From
\[
-\int_{\R} tx | \sigma _{\frac{1}{3} tN^{4} < \cdot < 3 tN^{4}} (x) \uep_{N} (t,x)|^2 dx <0
\]
taking $f=\sigma _{\frac{1}{3} tN^4 < \cdot < 3 tN^{4}} \uep_{N}$ in \eqref{eq:ueob}, we have
\[
tN^4 \| \sigma _{\frac{1}{3} tN^{4} < \cdot < 3 tN^{4}} \uep_{N} (t) \|_{L^2}
\lesssim \| u_N (t) \|_{\wt{X}}.
\]
\end{proof}

Second, by summing up the frequency localized estimates, we obtain the $L^2$-estimates.

\begin{cor} \label{cor:uL2}
For $t \ge 1$, we have
\begin{align}
& \sum_{k=0}^3 \sum_{l=0}^k \| t^{\frac{k+1}{4}} |x|^{-\frac{5k+1}{4}+l} \dx^l \uhp \|_{L^2} \lesssim \| u(t) \|_{\wt{X}}, \label{uhpL22} \\
& \sum_{k=0}^3 \| t^{\frac{k}{4}} |x|^{-\frac{k-3}{4}} \J_+ \dx^k \uhp \| _{L^2} \lesssim \| u (t) \|_{\wt{X}}, \label{uhpL2} \\
& \sum_{k=0}^3 \| t^{\frac{k+1}{5}} \lr{t^{-\frac{1}{5}} x}^{-\frac{k}{4}+1} \dx^k \uep \|_{L^2} \lesssim \| u(t) \|_{\wt{X}}. \label{uepL2}
\end{align}
\end{cor}

\begin{proof}
By \eqref{est:uhperr} and the definition of the $\wt{X}$-norm, we have
\begin{align*}
& \sum_{k=0}^3 \sum_{l=0}^k \| t^{\frac{k+1}{4}} |x|^{-\frac{5k+1}{4}+l} \dx^l \uhp \|_{L^2} \\
& \lesssim \sum_{k=0}^3 \sum_{l=0}^k \Bigg( \sum _{\substack{N \in 2^{\delta \mathbb{Z}} \\ N \ge t^{-\frac{1}{5}}}} \| t^{\frac{k+1}{4}} |x|^{-\frac{5k+1}{4}+l} \dx^l \uhp_N (t) \| _{L^2}^2 \Bigg)^{\frac{1}{2}} \\
& \quad + \sum_{k=0}^3 \sum_{l=0}^k \sum _{\substack{N \in 2^{\delta \mathbb{Z}} \\ N \ge t^{-\frac{1}{5}}}} \| (1-\PN) t^{\frac{k+1}{4}} |x|^{-\frac{5k+1}{4}+l} \dx^l \uhp_N \| _{L^2} \\
& \lesssim \sum_{k=0}^3 \sum_{l=0}^k \Bigg( \sum _{\substack{N \in 2^{\delta \mathbb{Z}} \\ N \ge t^{-\frac{1}{5}}}} \| t^{-k+l} N^{-5k+4l-1} \dx^l \uhp_N (t) \| _{L^2}^2 \Bigg)^{\frac{1}{2}} + \| u(t) \|_{\wt{X}} \\
& \lesssim \| u (t) \|_{\wt{X}}.
\end{align*}

We use \eqref{est:uhperr} and \eqref{uhpNl2} to obtain
\begin{align*}
& \sum_{k=0}^3 \| t^{\frac{k}{4}} |x|^{-\frac{k-3}{4}} \dx^k \J_+ \uhp \| _{L^2} \\
& \lesssim \sum_{k=0}^3 \Bigg( \sum _{\substack{N \in 2^{\delta \mathbb{Z}} \\ N \ge t^{-\frac{1}{5}}}} \| t^{\frac{k}{4}} |x|^{-\frac{k-3}{4}} \dx^k \J_+ \uhp_N (t) \| _{L^2}^2 \Bigg)^{\frac{1}{2}} \\
& \quad + \sum_{k=0}^3 \sum _{\substack{N \in 2^{\delta \mathbb{Z}} \\ N \ge t^{-\frac{1}{5}}}} \| (1-\PN) t^{\frac{k}{4}} |x|^{-\frac{k-3}{4}} \dx^k \J_+ \uhp_N (t) \| _{L^2} \\
& \lesssim \| u (t) \|_{\wt{X}}.
\end{align*}
This gives \eqref{uhpL2} with $k=0$.
For $k \ge 1$, because
\[
\J_+ \dx^k \uhp
= \dx^k (\J_+ \uhp) + t^{-\frac{k}{4}} |x|^{\frac{k-3}{4}} \sum_{l=0}^{k-1} C_{k,l} t^{\frac{(k-1)+1}{4}} |x|^{-\frac{5(k-1)+1}{4}+l} \dx^l \uhp,
\]
\eqref{uhpL22} yields \eqref{uhpL2}.

For the elliptic bound, we note that
\begin{equation} \label{uepw}
\uep = u^+_{\le t^{-\frac{1}{5}}} + \sum _{\substack{N \in 2^{\delta \mathbb{Z}} \\ N \ge t^{-\frac{1}{5}}}} \uep_N,
\quad u^+_{\le t^{-\frac{1}{5}}} := \sum _{\substack{N \in 2^{\delta \mathbb{Z}} \\ N \le t^{-\frac{1}{5}}}} P_N u^+.
\end{equation}
We use \eqref{est:ueperr} and \eqref{uepNl2} to obtain
\begin{align*}
\sum_{k=0}^3 \| t^{\frac{k+1}{5}} \dx^k \uep \|_{L^2}
& \lesssim t^{\frac{1}{5}} \| u^+_{\le t^{-\frac{1}{5}}} \|_{L^2} + \sum_{k=0}^3 \Bigg( \sum _{\substack{N \in 2^{\delta \mathbb{Z}} \\ N \ge t^{-\frac{1}{5}}}} (t^{\frac{k+1}{5}} N^k \| \uep_N \|_{L^2})^2 \Bigg)^{\frac{1}{2}} \\
& \quad + \sum_{k=0}^3 \sum _{\substack{N \in 2^{\delta \mathbb{Z}} \\ N \ge t^{-\frac{1}{5}}}} t^{\frac{k+1}{5}} \|  (1-\PN) \dx^k \uep_N \|_{L^2} \\
& \lesssim \| u(t) \|_{\wt{X}}.
\end{align*}
From $t^{\frac{k}{4}} |x|^{-\frac{k}{4}+1} N^k \lesssim |x| + tN^4$ and \eqref{est:ueperr2}, we have
\begin{align*}
t^{\frac{k}{4}} \| |x|^{-\frac{k}{4}+1} \dx^k \uep_N \|_{L^2 (|x| \ge t^{\frac{1}{5}})}
& \lesssim N^{-k} \| (|x|+tN^4) \sigma_{>t^{\frac{1}{5}}} (x) \dx^k \uep_N \|_{L^2} \\
& \lesssim \| (|x|+ tN^4) \uep_N \|_{L^2} + N^{-2} \| u_N \|_{L^2}.
\end{align*}
Because
\begin{align*}
t^{\frac{k}{4}} \| |x|^{-\frac{k}{4}+1} \dx^k \uep_{\le t^{-\frac{1}{5}}} \|_{L^2 (|x| \ge t^{\frac{1}{5}})} 
& \lesssim t^{\frac{k}{5}} \| x \dx^k \uep_{\le t^{-\frac{1}{5}}} \|_{L^2 (|x| \ge t^{\frac{1}{5}})} \\
& \lesssim \| x \uep_{\le t^{-\frac{1}{5}}} \|_{L^2 (|x| \ge t^{\frac{1}{5}})} + t^{\frac{1}{5}} \| \uep_{\le t^{-\frac{1}{5}}} \|_{L^2} \\
& \lesssim \| \uep_{\le t^{-\frac{1}{5}}} \|_{\wt{X}},
\end{align*}
by \eqref{est:ueperr2} and \eqref{uepNl2}, we obtain
\begin{align*}
& \| t^{\frac{k+1}{5}} (t^{-\frac{1}{5}} |x|)^{-\frac{k}{4}+1} \dx^k \uep \|_{L^2 (|x| \ge t^{\frac{1}{5}})} \\
& \lesssim \| \uep_{\le t^{-\frac{1}{5}}} \|_{\wt{X}} + \Bigg( \sum _{\substack{N \in 2^{\delta \mathbb{Z}} \\ N \ge t^{-\frac{1}{5}}}} \| (|x|+ tN^4) \uep_N \|_{L^2}^2 \Bigg)^{\frac{1}{2}} + \sum _{\substack{N \in 2^{\delta \mathbb{Z}} \\ N \ge t^{-\frac{1}{5}}}} N^{-2} \| u_N \|_{L^2} \\
& \quad + \sum _{\substack{N \in 2^{\delta \mathbb{Z}} \\ N \ge t^{-\frac{1}{5}}}} \| (1-\PN) t^{\frac{k+1}{5}} (t^{-\frac{1}{5}} |x|)^{-\frac{k}{4}+1} \sigma_{>t^{\frac{1}{5}}} (x) \dx^k \uep_N \|_{L^2} \\
& \lesssim \| u(t) \|_{\wt{X}}.
\end{align*}
\end{proof}

Finally, we show the pointwise decay estimates.

\begin{prop} \label{prop:est|u|}
For $t \ge 1$ and $k=0,1,2,3$, we have
\begin{align}
& | t^{\frac{k+1}{5}} \lr{t^{-\frac{1}{5}} x}^{-\frac{k-1}{4}} \dx^k \uhp (t,x) | \lesssim t^{-\frac{1}{10}} \| u (t) \|_{\wt{X}}, \label{est:uhpp} \\
& | t^{\frac{k+1}{5}} \lr{t^{-\frac{1}{5}} x}^{-\frac{k}{4}+\frac{7}{8}} \dx^k \uep (t,x) | \lesssim t^{-\frac{1}{10}} \| u (t) \|_{\wt{X}}. \label{est:uepp} 
\end{align}
\end{prop}

\begin{proof}
The Gagliardo-Nirenberg inequality
\[
|f| \lesssim \| f \|_{L^2}^{\frac{1}{2}} \| \dx f \|_{L^2}^{\frac{1}{2}}
\]
with $f = e^{-i\phi} \uhp_{N}$, $\dx (e^{-i\phi} \uhp) = t^{-\frac{1}{4}} e^{-i\phi} \J_+ \uhp$, Lemma \ref{lem:he_freq_est}, and \eqref{est:uhperr} imply
\begin{align*}
& | t^{\frac{k+1}{5}} \lr{t^{-\frac{1}{5}} x}^{-\frac{k-1}{4}} \dx^k \uhp_N (t,x) | \\
& \lesssim t^{\frac{2}{5}} N^{-k+1} \| \dx^k \uhp_N (t) \|_{L^{\infty}} \\
& \lesssim t^{\frac{11}{40}} N^{-k+1} \| \dx^k \uhp_{N} (t) \|_{L^2}^{\frac{1}{2}} \left\| \J_+ \dx^k \uhp_{N} (t) \right\|_{L^2}^{\frac{1}{2}} \\
& \lesssim t^{-\frac{1}{10}} \| N^{-1} u_{N} (t) \|_{L^2}^{\frac{1}{2}} \| t^{\frac{3}{4}} N^3 \J_+ \uhp_{N} (t) \|_{L^2}^{\frac{1}{2}} + t^{-\frac{1}{10}} \| u(t) \|_{\wt{X}} \\
& \lesssim t^{-\frac{1}{10}} \| u (t) \|_{\wt{X}}.
\end{align*}
Because $\uhp (t,x)$ is a finite sum of $\uhp_N(t,x)$'s, this yields the desired hyperbolic bound.

Next, we show the elliptic bound.
For $|x| \le t^{\frac{1}{5}}$, Bernstein's inequality implies
\[
| t^{\frac{k+1}{5}} \lr{t^{-\frac{1}{5}} x}^{-\frac{k}{4}+\frac{7}{8}} \dx^k u^+_{\le t^{-\frac{1}{5}}} (t,x) | 
\lesssim t^{\frac{1}{10}} \| u^+_{\le t^{-\frac{1}{5}}} (t) \|_{L^2}
\lesssim t^{-\frac{1}{10}} \| u_{\le t^{-\frac{1}{5}}} (t) \|_{\wt{X}}.
\]
Similarly, for $|x| \ge t^{\frac{1}{5}}$, we have
\begin{align*}
| t^{\frac{k+1}{5}} \lr{t^{-\frac{1}{5}} x}^{-\frac{k}{4}+\frac{7}{8}} \dx^k u^+_{\le t^{-\frac{1}{5}}} (t,x) |
& \lesssim t^{\frac{k}{5}} \| x \dx^k u^+_{\le t^{-\frac{1}{5}}} (t) \|_{L^{\infty}} \\
& \lesssim t^{-\frac{1}{10}} \| x u^+_{\le t^{-\frac{1}{5}}} (t) \|_{L^2} + t^{\frac{1}{10}} \| u_{\le t^{-\frac{1}{5}}} (t) \|_{L^2} \\
& \lesssim t^{-\frac{1}{10}} \| u_{\le t^{-\frac{1}{5}}} (t) \|_{\wt{X}}.
\end{align*}

For $|x|\le tN^4$, the Gagliardo-Nirenberg inequality, Lemma \ref{lem:he_freq_est}, and \eqref{est:ueperr} yield
\begin{align*}
| t^{\frac{k+1}{5}} \lr{t^{-\frac{1}{5}} x}^{-\frac{k}{4}+\frac{7}{8}} \dx^k \uep_N (t,x) | 
& \lesssim t^{\frac{9}{10}} N^{-k+\frac{7}{2}} | \sigma_{\le tN^4} \dx^k \uep_N (t,x) | \\
& \lesssim t^{-\frac{1}{10}} \| tN^4 \sigma_{\le tN^4} \uep_N (t) \|_{L^2} + t^{-\frac{1}{10}} \| u_N(t) \|_{\wt{X}} \\
& \lesssim t^{-\frac{1}{10}} \| u_N(t) \|_{\wt{X}}.
\end{align*}
For $|x| \ge tN^4$, there exists $M \in 2^{\delta \mathbb{N}_0}$ such that $\ue_N (t,x) = \sigma_{tM^4}(x) \ue_N (t,x)$.
Lemma \ref{lem:freq_spat_loc} and the same calculation as above lead
\begin{align*}
& | t^{\frac{k+1}{5}} \lr{t^{-\frac{1}{5}} x}^{-\frac{k}{4}+\frac{7}{8}} \dx^k \uep_N (t,x) | \\
& \lesssim t^{\frac{9}{10}} M^{-k+\frac{7}{2}} | \sigma_{tM^4}(x) \dx^k \uep_N (t,x) | \\
& \lesssim t^{\frac{9}{10}} M^{-k+\frac{7}{2}} N^{k+\frac{1}{2}} \| \sigma_{tM^4}(x) \uep_N (t) \|_{L^2} \\
& \quad + t^{-k+\frac{9}{10}} N^{\frac{1}{2}} M^{-5k+\frac{7}{2}} \max (1, tNM^4)^{-2} \| u_N (t) \|_{L^2}.
\end{align*}
Hence, we have
\begin{align*}
& \sum_{N>t^{-\frac{1}{5}}} | t^{\frac{k+1}{5}} \lr{t^{-\frac{1}{5}} x}^{-\frac{k}{4}+\frac{7}{8}} \dx^k \uep_N (t,x) | \\
& \lesssim \sum_{t^{-\frac{1}{5}}<N \le M} t^{-\frac{1}{10}} M^{-k-\frac{1}{2}} N^{k+\frac{1}{2}} \| x \uep_N (t) \|_{L^2} \\
& \quad + \sum_{N > M} t^{-\frac{1}{10}} M^{-k+\frac{7}{2}} N^{k-\frac{7}{2}} \| tN^4 \uep_N (t) \|_{L^2}
+ t^{-\frac{1}{10}} \| u(t) \|_{\wt{X}} \\
& \lesssim t^{-\frac{1}{10}} \| u(t) \|_{\wt{X}}.
\end{align*}
Combining these estimates with \eqref{uepw}, we obtain the desired elliptic bound.
\end{proof}

\begin{rmk}
For $t \ge 1$, the estimate
\[
| t^{\frac{k}{5}+\frac{3}{20}} \lr{t^{-\frac{1}{5}} x}^{-\frac{k}{4}+\frac{3}{8}} \dx^k \uhp (t,x) | \lesssim \| u(t) \|_{L^2} + t^{-\frac{1}{10}} \| u (t) \|_{\wt{X}}
\]
holds true.
Indeed,
\begin{align*}
& | t^{\frac{k}{5}+\frac{3}{20}} \lr{t^{-\frac{1}{5}} x}^{-\frac{k}{4}+\frac{3}{8}} \dx^k \uhp_N (t,x) | \\
& \lesssim t^{\frac{9}{20}} N^{-k+\frac{3}{2}} \| \dx^k \uhp_N (t) \|_{L^{\infty}} \\
& \lesssim t^{\frac{13}{40}} N^{-k+\frac{3}{2}}  \| \dx^k \uhp_{N} (t) \|_{L^2}^{\frac{1}{2}} \left\| \J_+ \dx^k \uhp_{N} (t) \right\|_{L^2}^{\frac{1}{2}} \\
& \lesssim t^{-\frac{1}{20}} \| u_{N} (t) \|_{L^2}^{\frac{1}{2}} \| t^{\frac{3}{4}} N^3 \J_+ \uhp_{N} (t) \|_{L^2}^{\frac{1}{2}} + t^{-\frac{1}{10}} \| u(t) \|_{\wt{X}} \\
& \lesssim \| u(t) \|_{L^2} + t^{-\frac{1}{10}} \| u (t) \|_{\wt{X}}.
\end{align*}
Accordingly, combining this with \eqref{est:uepp} and Remark \ref{rmk:reqH2}, we obtain \eqref{est:u_infty} at $t=1$.
\end{rmk}

\section{Wave packets} \label{S:wave_packet}

For $t \ge 1$ and $\rho \ge 0$, we define
\[
\Or := \{ v \in \R_- \colon |v| \gtrsim t^{-\frac{4}{5}+4\rho} \}.
\]
Setting
\[
\lambda := t^{-\frac{1}{2}} |v|^{-\frac{3}{8}},
\]
we define, for $v \in \Oo$,
\[
\Psi _v(t,x) := \chi \left( \lambda (x-vt) \right) e^{i \phi (t,x)},
\]
where $\chi$ is a sooth function with $\supp \chi \subset [-1+2^{-\delta}, 1-2^{-\delta}]$ and $\int_{\R} \chi (z) dz =1$, and $\phi$ is defined by \eqref{phase}.
The spatial support of $\Psi_v$ is included in $[2^{\delta}vt, \frac{vt}{2^{\delta}}]$.

For $v \in \Oo$,
\[
\partial_t \Psi _v (t,x)
= -\frac{x+vt}{2t} \lambda \chi' ( \lambda (x-vt)) e^{i\phi (t,x)} + i \partial_t \phi (t,x) \chi ( \lambda (x-vt) e^{i\phi (t,x)}.
\]
In the following calculation, we abbreviate $\dx^k (\chi (\lambda (x-vt))$ to $\dx^k \chi$.
By $\partial_t \phi = \frac{1}{5} (\dx \phi)^{5}$, we have
\begin{equation} \label{eq:LPsi}
\begin{aligned}
& (\Lc \Psi_v) (t,x) \\
& = -\frac{x+vt}{2t} \dx \chi e^{i\phi} + i \partial_t \phi \chi e^{i\phi} \\
& \quad -\frac{1}{5} \bigg\{ \chi \dx^{5} \left( e^{i \phi} \right) + 5 \dx \chi \dx^{4} \left( e^{i \phi} \right)  + 10 \dx^2 \chi \dx^{3} \left( e^{i \phi}\right)
+ 10 \dx^3 \chi \dx^{2} \left( e^{i \phi}\right) \\
& \quad + 5 \dx^4 \chi \dx \left( e^{i \phi} \right) + \dx^5 \chi e^{i \phi} \bigg\} \\
& = -\frac{x+vt}{2t} \dx \chi e^{i\phi} + i \partial_t \phi \chi e^{i\phi} \\
& \quad -\frac{1}{5} \bigg\{ \chi \Big( i (\dx \phi)^{5} + 10 (\dx \phi)^{3} \dx^2 \phi \Big)
+ 5 \dx \chi \Big( (\dx \phi)^{4} - 6i (\dx \phi)^{2} \dx^2 \phi \Big) \\
& \quad +10 \dx^2 \chi \Big( -i (\dx \phi)^{3} - 3 \dx \phi \dx^2 \phi \Big)
+ 10 \dx^3 \chi \Big( -(\dx \phi)^{2} \Big) \bigg\} e^{i\phi} \\
& \quad + O \left( t^{-2} |v|^{-\frac{5}{4}} \chi (\lambda (x-tv)) \right) \\
& = \frac{e^{i\phi}}{t \lambda} \dx \wt{\chi} + O \left( t^{-1} (t^{\frac{4}{5}} |v|)^{-\frac{5}{4}} \chi (\lambda (x-tv)) \right),
\end{aligned}
\end{equation}
where
\begin{align*}
\wt{\chi}(t,x) &:= \lambda \frac{x-vt}{2} \chi (\lambda (x-vt)) + 2i \lambda^2 t^{\frac{1}{4}} |x|^{\frac{3}{4}} \chi' (\lambda (x-vt)) \\
&\quad + 2 \lambda^3 t^{\frac{1}{2}} |x|^{\frac{1}{2}} \chi''(\lambda (x-vt))
\end{align*}
has the same localization of $\chi (\lambda (x-vt))$.

We show that $\Psi_v(t,x)$ is essentially frequency localized near $\xi_v := |v|^{\frac{1}{4}}$.
To state more precisely, for $v \in \Oo$, we define by $N_{v} \in 2^{\delta \mathbb{Z}}$ the nearest scaled dyadic number to $\xi_v$.
Then, $\frac{\xi_v}{2^{\delta}} < N_v < 2^{\delta} \xi_v$ holds.

\begin{lem} \label{lem:freq_psi}
For $t \ge 1$ and $v \in \Oo$, we have
\[
\mathcal{F}[ \Psi_v] (t,\xi)
= \frac{1}{2} \lambda^{-1} \chi_1 (\lambda^{-1} (\xi -\xi_v), \lambda^{-1} \xi_v) e^{\frac{1}{5}it \xi^5},
\]
where $\chi_1 (\cdot,a) \in \mathcal{S} (\R)$ for $a \ge 1$ and satisfies
\[
\int_{\R} \chi_1(\zeta, a) d\zeta = 1 + O \left( \frac{1}{a} \right)
\]
and
\begin{equation} \label{est:chi1p}
\sup_{a \ge 1} \sup_{ \zeta \in \R} |\lr{\zeta}^k \dz^l \chi_1 (\zeta, a)| \lesssim_{k,l} 1
\end{equation}
for any $k,l \in \N_0$.
\end{lem}

\begin{proof}
From Taylor's theorem, we can write
\begin{align*}
\phi (t,x)
& = \phi (t,vt) + \dx \phi (t,vt) (x-vt) + \frac{1}{2} \dx^2 \phi (t,vt) (x-vt)^2 \\
& \qquad + \int_{vt}^x \frac{(x-y)^2}{2} \dx^3 \phi (t,y) dy \\
& = \frac{\pi}{4} -\frac{4}{5} t \xi_v^5 + \xi_v (x-vt) - \frac{1}{8} \lambda^2 (x-vt)^2 + R \left( \lambda (x-vt), \lambda^{-1} \xi_v \right) ,
\end{align*}
where
\[
R(z,a) := -\frac{3}{32} \frac{z^3}{a} \int _0^1 (1-\theta)^2 \left( -\theta \frac{z}{a}+1 \right)^{-\frac{7}{4}} d\theta .
\]
We note that $R (z,a)$ is well-defined provided that $\max (z,0) < a$.
Changing variable $z= \lambda (x-vt)$, we have
\begin{align*}
& \mathcal{F} [\Psi_v] (t,\xi) \\
& = \frac{1}{\sqrt{2\pi}} \int_{\R} e^{-ix\xi} \chi ( \lambda (x-vt)) e^{i \phi (t,x)} dx \\
& = e^{i \frac{\pi}{4}} \lambda^{-1} e^{\frac{1}{5} it (5 \xi \xi_v^4 - 4 \xi_v^5)} \frac{1}{\sqrt{2\pi}} \int _{\R} e^{-i z \lambda^{-1} (\xi - \xi_v)} e^{-\frac{i}{8} z^2 + i R(z,  \lambda^{-1} \xi_v)} \chi (z) dz  \\
& = \frac{1}{2} \lambda^{-1} \chi_1 \left( \lambda^{-1} (\xi - \xi_v), \lambda^{-1} \xi_v \right) e^{\frac{1}{5}it\xi^5} ,
\end{align*}
where
\[
\chi_1 (\zeta, a) := \sqrt{2} (1+i) e^{-\frac{i}{5} (\zeta^5 a^{-3} + 5 \zeta^4 a^{-2} +10 \zeta^3 a^{-1} + 10\zeta^2)} \mathcal{F} [e^{-\frac{i}{8} z^2 + i R(z, a)} \chi ] (\zeta).
\]
By definition, $\chi_1 (\cdot ,a ) \in \mathcal{S}(\R)$ for $a \ge 1$.
From the Fresnel integrals, 
\[
(1+i) \sqrt{\frac{\pi}{2}} = \int_{\R} e^{i (\eta -\sqrt{2} (\zeta+\frac{z}{4}))^2} d\eta
\]
holds for any $z, \zeta \in \R$.
Accordingly, we have
\begin{align*}
\mathcal{F} [e^{-\frac{i}{8} z^2 + i R(z, a)} \chi ] (\zeta) 
& = e^{2i\zeta^2} \frac{1}{\sqrt{2\pi}} \int_{\R} e^{-2i (\zeta+\frac{z}{4})^2 + iR(z,a)} \chi (z) dz \\
& = \frac{1}{1+i} \sqrt{\frac{2}{\pi}} e^{2i\zeta^2} \frac{1}{\sqrt{2\pi}} \int_{\R} \int_{\R} e^{i(\eta^2 -2 \sqrt{2} \eta (\zeta + \frac{z}{4}))} e^{iR(z,a)} \chi (z) d\eta dz \\
& = \frac{1}{1+i} \sqrt{\frac{2}{\pi}} e^{2i\zeta^2} \int_{\R} e^{i\eta^2} e^{-2\sqrt{2} i \eta \zeta} \wh{\chi_2} \left( \frac{\eta}{\sqrt{2}}, a \right) d\eta,
\end{align*}
where
\[
\chi_2 (\cdot, a) := \chi e^{iR(\cdot,a)} \in \mathcal{S} (\R)
\]
for $a \ge 1$.
Hence, we can write
\begin{equation} \label{eq:chi1}
\chi_1 (\zeta, a)
= 2 \sqrt{\frac{2}{\pi}} e^{-\frac{i}{5} (\zeta^5 a^{-3} + 5 \zeta^4 a^{-2} +10 \zeta^3 a^{-1})} \int_{\R} e^{2i\eta^2} e^{-4 i \eta \zeta} \wh{\chi_2} \left( \eta, a \right) d\eta.
\end{equation}
As $\chi_2 (\cdot, a) \in \mathcal{S} (\R)$, $e^{-\frac{i}{5} (\zeta^5 a^{-3} + 5 \zeta^4 a^{-2} +10 \zeta^3 a^{-1})} = 1 + O \left( \frac{|\zeta|^3}{a} \lr{\zeta}^2 \right)$, and
\begin{equation} \label{eq:chi2}
\sup_{a \ge 1} \sup_{\eta \in \R} |\lr{\eta}^k \de^l \wh{\chi_2} (\eta, a)| \lesssim_{k,l} 1,
\end{equation}
we obtain
\[
\int_{\R} \chi_1 (\zeta,a) d\zeta
= \sqrt{2\pi} \wh{\chi_2} (0,a) +O \left( \frac{1}{a} \right).
\]
Because $e^{i R(z,a)} = 1 + O (\frac{1}{a})$ for $|z| < 1$ and $a \ge 1$, we have
\[
\sqrt{2\pi} \wh{\chi_2} (0,a) = \int_{\R} \chi (z) e^{iR(z,a)} dz = 1+ O \left( \frac{1}{a} \right).
\]
Finally, \eqref{est:chi1p} follows from \eqref{eq:chi1} and \eqref{eq:chi2}.
\end{proof}

The integration by parts with \eqref{est:chi1p} yields
\[
\left| \Big( 1-P_{\frac{N_v}{2^{\delta}} \le \cdot \le 2^{\delta}N_v}^+ \Big) \Psi_v(t,x) \right|
\lesssim_l \min (1, |x|^{-1} t|v|)^2 (\lambda \xi_v^{-1})^l
\]
for any $l \ge 0$, which implies
\begin{equation}
\label{Psi_L1}
\left\| \Big( 1-P_{\frac{N_v}{2^{\delta}} \le \cdot \le 2^{\delta}N_v}^+ \Big) \Psi_v(t) \right\|_{L_x^1}
\lesssim_c t^{\frac{1}{5}} (t^{\frac{4}{5}} |v|)^{-c}
\end{equation}
for $v \in \Oo$ and any $c \ge 0$.

For $v \in \Oo$, we define
\[
\gamma (t,v) := \int_{\R} u (t,x) \overline{\Psi}_v (t,x) dx,
\]
which is a good approximation of $u$ in the following sense.

\begin{prop} \label{prop:approx}
For $t \ge 1$ and $k=0,1,2,3$, we have the bounds
\begin{gather}
\| t^{\frac{k+1}{5}} (t^{\frac{4}{5}} |v|)^{-\frac{k}{4}+\frac{9}{16}} \big( \dx^k u^+ (t,vt) - i^k \lambda |v|^{\frac{k}{4}} e^{i \phi (t,vt)} \gamma (t,v) \big) \|_{L^{\infty}_v (\Oo)}
\lesssim t^{-\frac{1}{10}} \| u(t) \|_{\wt{X}}, \label{ugsp1}
\\
\| t^{\frac{k+3}{5}} (t^{\frac{4}{5}} |v|)^{-\frac{k}{4}+\frac{3}{8}} \big( \dx^k u^+ (t,vt) -  i^k \lambda |v|^{\frac{k}{4}} e^{i \phi (t,vt)} \gamma (t,v) \big) \|_{L^2_v (\Oo)}
\lesssim t^{-\frac{1}{10}} \| u(t) \|_{\wt{X}}. \label{ugsl2}
\end{gather}
Moreover, in the frequency space, we have
\begin{gather*}
\| (t^{\frac{4}{5}} |v|)^{\frac{3}{16}} \big( \wh{u} (t,\xi_v) - 2 e^{\frac{1}{5} it \xi_v^5} \gamma (t,v) \big) \|_{L^{\infty}_v (\Oo)}
\lesssim t^{-\frac{1}{10}} \| u(t) \|_{\wt{X}},
\\
\| t^{\frac{2}{5}} \big( \wh{u} (t,\xi_v) - 2 e^{\frac{1}{5} it \xi_v^5} \gamma (t,v) \big) \|_{L^2_v (\Oo)}
\lesssim t^{-\frac{1}{10}} \| u(t) \|_{\wt{X}}.
\end{gather*}
\end{prop}

\begin{proof}
First, we show that
\begin{equation} \label{eq:L^2v}
\Big\| |v|^{-\frac{3}{8}} \int_{\R} f(t,x) \chi (\lambda (x-vt)) dx \Big\|_{L^2_v (\Oo)}
\lesssim \| f(t,x) \|_{L^2_x (t^{-\frac{1}{5}} |x| \gtrsim 1)}
\end{equation}
holds true.
By changing variable $z = \lambda (x-vt)$,
\[
\text{L.H.S. of \eqref{eq:L^2v}}
= t^{\frac{1}{2}} \Big\| \int_{\R} f(t, t^{\frac{1}{2}} |v|^{\frac{3}{8}} z + vt) \chi (z) dz \Big\|_{L^2_v}.
\]
Setting $\wt{v} = t^{\frac{1}{2}} |v|^{\frac{3}{8}} z + vt$, we note
\begin{align*}
& t^{-\frac{1}{5}} |\wt{v}| = t^{\frac{4}{5}} |v| \Big\{ 1- (t^{\frac{4}{5}} |v|)^{-\frac{5}{8}} z \Big\} \gtrsim 1, \\
& \frac{d \wt{v}}{dv} = t \Big\{ 1-\frac{3}{8} (t^{\frac{4}{5}} |v|)^{-\frac{5}{8}} z \Big\} \gtrsim t,
\end{align*}
for $v \in \Oo$.
Then, we have
\[
\text{L.H.S. of \eqref{eq:L^2v}}
\lesssim \Big\| \int_{\R} f(t, \wt{v}) \chi (z) dz \Big\|_{L^2_{\wt{v}}( t^{-\frac{1}{5}} |\wt{v}|\gtrsim 1)}
\lesssim \| f(t,x) \|_{L^2_x (t^{-\frac{1}{5}} |x| \gtrsim 1)}.
\]

Proposition \ref{prop:est|u|} and \eqref{Psi_L1} imply
\begin{align*}
& \Big| \int_{\R} (\uep (t,x) + \ol{u^+} (t,x)) \overline{\Psi}_v (t,x) dx \Big| \\
& \lesssim \lambda^{-1} (t^{\frac{4}{5}}|v|)^{-\frac{7}{8}} \| (t^{-\frac{1}{5}} |x|)^{\frac{7}{8}} \uep (t) \|_{L^{\infty}} + \| u^+(t) \|_{L_x^{\infty}} \| P^- \Psi_v (t) \|_{L^1} \\
& \lesssim (t^{\frac{4}{5}} |v|)^{-\frac{1}{2}} \cdot t^{-\frac{1}{10}} \| u(t) \|_{\wt{X}}
\end{align*}
for $v \in \Oo$.
From \eqref{Psi_L1}, \eqref{eq:L^2v}, \eqref{uepL2}, and Proposition \ref{prop:est|u|}, we have
\begin{align*}
& \Big\| \int_{\R} (\uep (t,x) + \ol{u^+} (t,x)) \overline{\Psi}_v (t,x) dx \Big\|_{L^2_v (\Oo)} \\
& \lesssim t^{-\frac{3}{8}} \| |x|^{\frac{3}{8}} \uep (t) \|_{L^2_x} + \| (t^{-\frac{1}{5}} |x|)^{\frac{1}{4}} u^+(t) \|_{L_x^{\infty}} \big\| (t^{\frac{4}{5}} |v|)^{-\frac{1}{4}} \| P^- \Psi_v (t) \|_{L^1_x} \big\|_{L^2_v (\Oo)} \\
& \lesssim t^{-\frac{3}{10}} \| \lr{t^{-\frac{1}{5}} x}^{\frac{3}{8}} \uep (t) \|_{L^2} + t^{-\frac{2}{5}} \cdot t^{-\frac{1}{10}} \| u(t) \|_{\wt{X}} \\
& \lesssim t^{-\frac{2}{5}} \cdot t^{-\frac{1}{10}} \| u(t) \|_{\wt{X}}.
\end{align*}
Hence, from $\lambda |v|^{\frac{k}{4}} = t^{-\frac{k+1}{5}} (t^{\frac{4}{5}} |v|)^{\frac{k}{4}-\frac{3}{8}}$, we obtain
\[
i^k \lambda |v|^{\frac{k}{4}} \gamma (t,v)
= i^k \lambda |v|^{\frac{k}{4}} \int_{\R} \uhp (t,x) \overline{\Psi}_v (t,x) dx + R_k(t,v),
\]
where $R_k$ is a function satisfying
\begin{align*}
& \| t^{\frac{k+1}{5}} (t^{\frac{4}{5}} |v|)^{-\frac{k}{4}+\frac{9}{16}} R_k(t,v) \|_{L^{\infty}_v (\Oo)} \lesssim t^{-\frac{1}{10}} \| u(t) \|_{\wt{X}}, \\
& \| t^{\frac{k+3}{5}} (t^{\frac{4}{5}} |v|)^{-\frac{k}{4}+\frac{3}{8}} R_k (t,v) \|_{L^2_v (\Oo)}
\lesssim t^{-\frac{1}{10}} \| u(t) \|_{\wt{X}}.
\end{align*}

Here, \eqref{est:uhpp} yields
\begin{align*}
& |v|^{-\frac{k-1}{4}} \left| \int_{\R} t^{\frac{1}{4}} \left( |x|^{-\frac{1}{4}} - (|v|t)^{-\frac{1}{4}} \right) \dx^k \uhp (t,x)  \ol{\Psi}_v(t,x) dx \right| \\
& \lesssim t^{\frac{k}{5}} (t^{\frac{4}{5}} |v|)^{-\frac{k}{4}-\frac{5}{8}} \int_{\R} |\dx^k \uhp (t,x)| \chi (\lambda (x-vt)) dx \\
& \lesssim (t^{\frac{4}{5}} |v|)^{-\frac{1}{2}} \cdot t^{-\frac{1}{10}} \| u(t) \|_{\wt{X}}.
\end{align*}
By \eqref{eq:L^2v} and \eqref{uhpL22}, we have
\begin{align*}
& \left\| |v|^{-\frac{k-1}{4}} \int_{\R} t^{\frac{1}{4}} \left( |x|^{-\frac{1}{4}} - (|v|t)^{-\frac{1}{4}} \right) \dx^k \uhp (t,x)  \ol{\Psi}_v(t,x) dx \right\|_{L^2_v(\Oo)} \\
& \lesssim t^{-\frac{1}{2}} \left\| |v|^{-\frac{k}{4}-\frac{5}{8}} \int_{\R} \dx^k \uhp (t,x) \chi (\lambda (x-vt)) dx \right\|_{L^2_v(\Oo)} \\
& \lesssim t^{-\frac{1}{2}} \bigg\| \Big( \frac{|x|}{t} \Big)^{-\frac{k+1}{4}} \dx^k \uhp \bigg\|_{L^2_v(\Oo)} \\
& \lesssim t^{-\frac{2}{5}} \cdot t^{-\frac{1}{10}} \| u(t) \|_{\wt{X}}.
\end{align*}
Moreover, \eqref{uhpL2} implies
\begin{align*}
& \left| |v|^{-\frac{k}{4}}\int_{\R} t^{\frac{1}{4}} |x|^{-\frac{1}{4}} \dx (e^{-i \phi} \dx^k\uhp) (t,x) \chi (\lambda (x-vt)) dx \right| \\
& \lesssim t^{\frac{k}{4}} (t|v|)^{-1} \lambda^{-\frac{1}{2}} \| |x|^{-\frac{k-3}{4}} \J_+ \dx^k \uhp \|_{L^2} \\
& \lesssim (t^{\frac{4}{5}} |v|)^{-\frac{13}{16}} \cdot t^{-\frac{1}{10}} \| u(t) \|_{\wt{X}}
\end{align*}
and
\begin{align*}
& \left\| |v|^{-\frac{k}{4}}\int_{\R} t^{\frac{1}{4}} |x|^{-\frac{1}{4}} \dx (e^{-i \phi} \dx^k\uhp) (t,x) \chi (\lambda (x-vt)) dx \right\|_{L^2_v(\Oo)} \\
& \lesssim \| (t^{\frac{4}{5}} |v|)^{-\frac{13}{16}} \|_{L^2_v(\Oo)} \cdot t^{-\frac{1}{10}} \| u(t) \|_{\wt{X}} \\
& \lesssim t^{-\frac{2}{5}} \cdot t^{-\frac{1}{10}} \| u(t) \|_{\wt{X}}.
\end{align*}
Because
\begin{align*}
& \uhp (t,x) \ol{\Psi_v} (t,x) \\
&= -i |v|^{-\frac{1}{4}} \dx \uhp (t,x) \ol{\Psi}_v(t,x) - i t^{\frac{1}{4}} \left( |x|^{-\frac{1}{4}}- (|v|t)^{-\frac{1}{4}} \right) \dx \uhp (t,x) \ol{\Psi}_v (t,v) \\
& \quad +i t^{\frac{1}{4}} |x|^{-\frac{1}{4}} \dx ( \uhp e^{-i\phi}) (t,x) \chi (\lambda (x-vt)),
\end{align*}
we can write
\[
i^k \lambda |v|^{\frac{k}{4}} \gamma (t,v) 
= \lambda \int_{\R} \dx^k\uhp (t,x) \overline{\Psi}_v (t,x) dx + R_k (t,v).
\]
We set $w_k (t,x) := e^{-i\phi (t,x)} \dx^k \uhp (t,x)$.
By \eqref{uepL2}, \eqref{eq:L^2v}, and \eqref{est:uepp}, we have
\begin{align*}
& \dx^k u^+ (t,vt) - i^k \lambda |v|^{\frac{k}{4}} e^{i \phi (t,vt)} \gamma (t,v) \\
& = \dx^k \uhp (t,vt) - \lambda e^{i\phi (t,vt)} \int_{\R} \dx^k \uhp (t,x) \ol{\Psi}_v(t,v) dx + R_k (t,v) \\
& = e^{i\phi (t,vt)} \lambda \int_{\R} \left( w_k (t,vt) - w_k (t,x) \right) \chi (\lambda (x-vt)) dx + R_k (t,v).
\end{align*}
Changing variable $z = \lambda (x-vt)$ and \eqref{uhpL2} imply
\begin{equation} \label{eq:diffw}
\begin{aligned}
& \int_{\R} |w_k (t,vt) - w_k (t,x) | \chi (\lambda (x-vt)) dx \\
& \le \lambda^{-1} \int_{\R} |w_k (t,vt) - w_k (t, \lambda^{-1} z +vt)| \chi (z) dz \\
& = \lambda^{-2} \int_{\R} \left| \int_0^1 \dx w_k (t, vt+(1-\theta) \lambda^{-1} z) d\theta \right| |z| \chi (z) dz \\
& \lesssim t^{-\frac{1}{4}} \lambda^{-\frac{3}{2}} (t|v|)^{\frac{k-3}{4}} \left\| |x|^{-\frac{k-3}{4}} \J_+ \dx^k \uhp (t) \right\|_{L^2} \\
& \lesssim t^{-\frac{k}{5}} (t^{\frac{4}{5}} |v|)^{\frac{k}{4}-\frac{3}{16}} \cdot t^{-\frac{1}{10}} \| u(t) \|_{\wt{X}}.
\end{aligned}
\end{equation}
Similarly, Young's inequality, changing variables $z = \lambda (x-vt)$ and $\wt{v} = vt+(1-\theta) \lambda^{-1} z$, and \eqref{uhpL2} give
\begin{equation} \label{eq:diffwl2}
\begin{aligned}
& \left\| t^{\frac{k+3}{5}} (t^{\frac{4}{5}} |v|)^{-\frac{k}{4}+\frac{3}{8}} \lambda \int_{\R} |w_k (t,vt) - w_k (t,x) | \chi (\lambda (x-vt)) dx \right\|_{L^2_v (\Oo)} \\
& \le t^{\frac{23}{20}} \int_{\R} \int _0^1 \left\| |v|^{\frac{3-k}{4}} | (\J_+ \dx^k \uhp)(t, vt+(1-\theta) \lambda^{-1} z) | \right\|_{L^2_v (\Oo)} \chi (z) d\theta dz  \\
& \le t^{\frac{k}{4}-\frac{1}{10}} \left\| |\wt{v}|^{\frac{3-k}{4}} (\J_+ \dx^k \uhp)(t,\wt{v}) \right\|_{L^2_{\wt{v}}} \\
& \lesssim t^{-\frac{1}{10}} \| u(t) \|_{\wt{X}}.
\end{aligned}
\end{equation}

Next, we consider the frequency space approximation.
Proposition \ref{prop:est|u|} and Lemma \ref{lem:freq_psi} yield
\begin{align*}
& \wh{u}(t,\xi_v) - 2 e^{\frac{1}{5} it\xi_v^5} \gamma (t,v) \\
& = e^{\frac{1}{5} it\xi_v^5} \int_{\R} \big( \wh{u}(t,\xi_v) e^{-\frac{1}{5} it\xi_v^5} - \wh{u}(t,\xi) e^{-\frac{1}{5} it\xi^5} \big) \lambda^{-1} \ol{\chi_1 (\lambda^{-1} (\xi-\xi_v), \lambda^{-1}\xi_v)} d\xi \\
& \quad + O \Big( (t^{\frac{4}{5}} |v|)^{-\frac{5}{8}} | \wh{u} (t,\xi_v)| \Big).
\end{align*}
By changing variable $\zeta = \lambda^{-1} (\xi-\xi_v)$, we have
\begin{align*}
& \left| \int_{\R} \big( \wh{u}(t,\xi_v) e^{-\frac{1}{5} it\xi_v^5} - \wh{u}(t,\xi) e^{-\frac{1}{5} it\xi^5} \big) \lambda^{-1} \ol{\chi_1 (\lambda^{-1} (\xi-\xi_v), \lambda^{-1}\xi_v)} d\xi \right| \\
& \le \int_{\R} |\xi-\xi_v| \int_0^1 | \wh{\J u} (\theta (\xi_v-\xi) + \xi)| d\theta \lambda^{-1} |\chi_1 (\lambda^{-1} (\xi-\xi_v), \lambda^{-1}\xi_v)| d\xi \\
& = \lambda \int_{\R} \int_0^1 | \wh{\J u} ( \xi_v + \lambda \zeta (1-\theta)) | d\theta |\zeta \chi_1 (\zeta, \lambda^{-1} \xi_v)| d\zeta.
\end{align*}
Because $|\wh{u}(t,\xi_v)| \le |\wh{u}(t,\xi_v) - 2 e^{\frac{1}{5} it\xi_v^5} \gamma (t,v)| + 2 |\gamma (t,v)|$ and $\chi_1 (\cdot ,a) \in \mathcal{S} (\R)$ for $a \ge 1$, we have
\begin{align*}
\Big| \wh{u}(t,\xi_v) - 2 e^{\frac{1}{5} it\xi_v^5} \gamma (t,v) \Big|
& \lesssim \lambda^{\frac{1}{2}} \| \J u \|_{L^2} + (t^{\frac{4}{5}} |v|)^{-\frac{5}{8}} | \gamma (t,v) | \\
& \lesssim (t^{\frac{4}{5}} |v|)^{-\frac{3}{16}} \cdot t^{-\frac{1}{10}} \| u(t) \|_{\wt{X}}.
\end{align*}
As in \eqref{eq:diffwl2}, Young's inequality and changing variable $\wt{\mathrm{v}} = \xi_v + \lambda \zeta (1-\theta)$ yield
\begin{align*}
& \left\| \wh{u}(t,\xi_v) - 2 e^{\frac{1}{5} it\xi_v^5} \gamma (t,v) \right\|_{L^2_v (\Oo)} \\
& \lesssim \left\| \lambda \int_{\R} \int_0^1 | \wh{\J u} ( \xi_v + \lambda \zeta (1-\theta)) | d\theta |\zeta \chi_1 (\zeta, \lambda^{-1} \xi_v)| d\zeta \right\|_{L^2_v (\Oo)} \\
& \quad + \| (t^{\frac{4}{5}} |v|)^{-\frac{5}{8}} \wh{u} (t,\xi_v) \|_{L^2_v(\Oo)} \\
& \le t^{-\frac{1}{2}} \left\| (\J u) (t,\wt{\mathrm{v}}) \right\|_{L^2_{\wt{\mathrm{v}}}} + t^{-\frac{1}{2}} \| |\xi|^{-1} \wh{u}(t) \|_{L^2(t^{\frac{1}{5}} |\xi| \gtrsim 1)} \\
& \lesssim t^{-\frac{2}{5}} \cdot t^{-\frac{1}{10}} \| u(t) \|_{\wt{X}}.
\end{align*}
\end{proof}

\section{Proof of the main theorem} \label{S:proof}

We derive an ordinary differential equation with respect to $\gamma$.

\begin{prop} \label{prop:gamma_decay}
Let $u$ be a solution to \eqref{5mKdV} satisfying \eqref{est:u_infty0}.
Then, for $t \ge 1$, we have
\begin{align*}
& \| t (t^{\frac{4}{5}} |v|)^{\frac{3}{16}} \big( \dot{\gamma} (t) + 3i \alpha t^{-1} |\gamma (t)|^2 \gamma (t) \big) \|_{L_v^{\infty} (\Oo)}
\lesssim \eps, \\
& \| t^{\frac{7}{5}} \big( \dot{\gamma} (t) + 3i \alpha t^{-1} |\gamma (t)|^2 \gamma (t) \big) \|_{L_v^2 (\Oo)}
\lesssim \eps,
\end{align*}
where the implicit constants are independent of $D$ and $T$.
\end{prop}

\begin{proof}
We use $\err$ to denote error terms that satisfy the estimates
\[
\| t (t^{\frac{4}{5}} |v|)^{\frac{3}{16}} \err \|_{L^{\infty}_v (\Oo)} \lesssim \eps, \quad
\| t^{\frac{7}{5}} \err \|_{L^2_v (\Oo)} \lesssim \eps.
\]
Then, for any $k=0,1,2,3$, we have
\begin{equation} \label{est:errue}
t^{-1} |v|^{-\frac{k}{4}} \int_{\R} |\dx^k \uep (t,x)| \chi (\lambda (x-vt)) dx = \err.
\end{equation}
Indeed, \eqref{est:uepp} shows
\begin{align*}
& t^{-1} |v|^{-\frac{k}{4}} \int_{\R} |\dx^k \uep (t,x)| \chi (\lambda (x-vt)) dx \\
& \lesssim t^{-1} (t^{\frac{4}{5}} |v|)^{-\frac{1}{2}} \sup_{x \in \R} \left| t^{\frac{k+1}{5}} \lr{t^{-\frac{1}{5}} x}^{-\frac{k}{4}+\frac{7}{8}} \dx^k \uep (t,x)  \right| \\
& \lesssim t^{-1} (t^{\frac{4}{5}} |v|)^{-\frac{1}{2}} \cdot t^{-\frac{1}{10}} \| u (t) \|_{\wt{X}}.
\end{align*}
From \eqref{uepL2} and \eqref{eq:L^2v}, we have
\begin{align*}
& \left\| t^{-1} |v|^{-\frac{k}{4}} \int_{\R} |\dx^k \uep (t,x)| \chi (\lambda (x-vt)) dx \right\|_{L^2_v(\Oo)} \\
& \lesssim t^{-\frac{3}{2}} \| t^{\frac{k+1}{5}} \lr{t^{-\frac{1}{5}} x}^{-\frac{k}{4}+\frac{3}{8}} \dx^k \uep (t) \|_{L^2} \\
& \lesssim t^{-\frac{7}{5}} \cdot t^{-\frac{1}{10}} \| u (t) \|_{\wt{X}}.
\end{align*}
Owing to Lemma \ref{lem:Xtilde}, they are error terms.

Because $e^{i\phi (t,x)} \wt{\chi}$ has the same localization property as $\Psi_v(t,x)$, from \eqref{eq:LPsi}, \eqref{Psi_L1}, and Proposition \ref{prop:est|u|}, we have
\begin{align*}
& \left| \int_{\R} u \Lc \ol{\Psi}_v (t,x) dx \right| \\
& \lesssim t^{-\frac{4}{5}} (t^{\frac{4}{5}} |v|)^{\frac{3}{8}} \left| \int_{\R} e^{-i\phi (t,x)} u^+ (t,x) \dx \overline{\wt{\chi}} (t,x) dx \right| \\
& \quad + t^{-\frac{4}{5}} (t^{\frac{4}{5}} |v|)^{\frac{3}{8}} \left| \int_{\R} e^{-i\phi (t,x)} \ol{u^+ (t,x)} \dx \overline{\wt{\chi}} (t,x) dx \right| \\
& \quad + t^{-1} (t^{\frac{4}{5}} |v|)^{-\frac{5}{4}} \int_{\R} |u(t,x) \chi (\lambda (x-vt))| dx \\
& \lesssim t^{-\frac{21}{20}} (t^{\frac{4}{5}} |v|)^{\frac{3}{8}} \int_{\R} |\J_+ \uhp (t,x) \wt{\chi} (t,x)| dx
+ t^{-1} \int_{\R} |\uep (t,x) \wt{\chi} (t,x)| dx \\
& \quad + t^{-1} (t^{\frac{4}{5}} |v|)^{-\frac{9}{8}} \cdot t^{-\frac{1}{10}} \| u (t) \|_{\wt{X}}.
\end{align*}
From \eqref{uhpL2}, we obtain
\begin{align*}
& t^{-\frac{21}{20}} (t^{\frac{4}{5}} |v|)^{\frac{3}{8}} \int_{\R} |\J_+ \uhp (t,x) \wt{\chi} (t,x)| dx \\
& \lesssim t^{-\frac{6}{5}} (t^{\frac{4}{5}} |v|)^{-\frac{3}{8}}  \| |x|^{\frac{3}{4}} \J_+ \uhp (t) \|_{L^2} \| \wt{\chi} (t) \|_{L^2} \\
& \lesssim t^{-1} (t^{\frac{4}{5}} |v|)^{-\frac{3}{16}} \cdot t^{-\frac{1}{10}} \| u(t) \|_{\wt{X}}.
\end{align*}
Similarly, \eqref{uhpL2} and \eqref{eq:L^2v} imply
\begin{align*}
\left\| t^{-\frac{21}{20}} (t^{\frac{4}{5}} |v|)^{\frac{3}{8}} \int_{\R} |\J_+ \uhp (t,x) \wt{\chi} (t,x)| dx \right\|_{L^2_v (\Oo)}
& \lesssim t^{-\frac{3}{2}} \| |x|^{\frac{3}{4}} \J_+ \uhp \|_{L^2} \\
& \lesssim t^{-\frac{7}{5}} \cdot t^{-\frac{1}{10}} \| u(t) \|_{\wt{X}}.
\end{align*}

We note that
\begin{equation} \label{intpsi}
\dx \Psi_v(t,v) = i|v|^{\frac{1}{4}} \Psi_v(t,x) + \lambda \wt{\Psi}_v(t,x),
\end{equation}
where $\wt{\Psi}_v(t,v) = \big\{ \chi' (\lambda (x-vt)) +i t^{-\frac{1}{4}} (|x|^{\frac{1}{4}}- (t|v|)^{\frac{1}{4}}) \lambda^{-1} \chi (\lambda (x-vt)) \big\} e^{i \phi (t,x)}$ has the same localization of $\Psi_v(t,x)$.
Hence, we can write
\begin{align*}
\dot{\gamma}(t,v)
& = \int_{\R} ( \Lc u \cdot \overline{\Psi}_v + u \Lc \overline{\Psi}_v ) (t,x) dx \\
& = \int_{\R} \dx \Big\{ \alpha \Big( 2u^2 \dx^2u + 3u (\dx u)^2 \Big) + \beta u^5 \Big\} \ol{\Psi}_v (t,x) dx + \err. \\
& = i \alpha |v|^{\frac{1}{4}} \int_{\R} \Big( 2u^2 \dx^2u + 3u (\dx u)^2 \Big) \ol{\Psi}_v (t,x) dx \\
& \quad - i \alpha \lambda \int_{\R} \Big( 2u^2 \dx^2u + 3u (\dx u)^2 \Big) \ol{\wt{\Psi}}_v (t,x) dx \\
& \quad + 5 \beta \int_{\R} u^4 \dx u \ol{\Psi}_v(t,x) dx + \err.
\end{align*}
The bootstrap assumption \eqref{est:u_infty0} yields
\begin{gather*}
\lambda \left| \int_{\R} \Big( 2u^2 \dx^2u + 3u (\dx u)^2 \Big) \ol{\wt{\Psi}}_v (t,x) dx \right|
\lesssim t^{-1} (t^{\frac{4}{5}} |v|)^{-\frac{5}{8}} (D\eps)^3
\lesssim t^{-1} (t^{\frac{4}{5}} |v|)^{-\frac{5}{8}} \eps,
\\
\left| \int_{\R} u^4 \dx u \ol{\Psi}_v(t,x) dx \right|
\lesssim t^{-1} (t^{\frac{4}{5}} |v|)^{-\frac{5}{4}} (D\eps)^5
\lesssim t^{-1} (t^{\frac{4}{5}} |v|)^{-\frac{5}{4}} \eps.
\end{gather*}
We therefore arrive at
\[
\dot{\gamma}(t) = i \alpha |v|^{\frac{1}{4}} \int_{\R} \Big( 2u^2 \dx^2u + 3u (\dx u)^2 \Big) \ol{\Psi}_v (t,x) dx + \err.
\]
We divide $u$ into $u= u^+ + \ol{u^+}$ and $u^+ = \uhp+\uep$.
If at least one of $u$ on the right hand side is $\uep$ or $\ol{\uep}$, then the right hand side is an error term because of \eqref{est:u_infty0} and \eqref{est:errue}. 
Accordingly, setting $\uh = \uhp + \ol{\uhp}$, we have
\begin{align*}
\dot{\gamma}(t) & = i \alpha |v|^{\frac{1}{4}} \int_{\R} \Big\{ 2(\uh)^2 \dx^2 \uh + 3 \uh (\dx \uh)^2 \Big\} \ol{\Psi}_v (t,x) dx + \err.
\end{align*}
Here, we observe that for $\frac{|v|}{2^{\delta}} \le \frac{|x|}{t} \le 2^{\delta}|v|$,
\begin{align*}
\uhp
& = \sum_{\substack{N \in 2^{\delta \mathbb{Z}} \\ \frac{N_v}{\sqrt{3} 2^{2\delta}} \le N \le \sqrt{3} 2^{2\delta} N_v}} \uhp_N \\
& = P_{\frac{N_v}{\sqrt{3} 2^{3\delta}} \le \cdot \le \sqrt{3} 2^{3\delta} N_v}^+ \uhp + \sum_{\substack{N \in 2^{\delta \mathbb{Z}} \\ \frac{N_v}{\sqrt{3} 2^{2\delta}} \le N \le \sqrt{3} 2^{2\delta} N_v}} (1-P_{\frac{N}{2^{\delta}} \le \cdot \le 2^{\delta} N}^+) \uhp_N .
\end{align*}
If the frequency supports of
\begin{align*}
&
\begin{aligned}
& (\uh)^2 \dx^2 \uh - 2 |\uhp|^2 \dx^2 \uhp - (\uhp)^2 \ol{\dx^2\uhp} \\
&
= (\uhp)^2 \dx^2 \uhp + 2 |\uhp|^2 \ol{\dx^2 \uhp} + \ol{\uhp}^2 \dx^2 \uhp + \ol{\uhp}^2 \ol{\dx^2 \uhp},
\end{aligned}
\\
&
\begin{aligned}
& \uh (\dx \uh)^2 - \ol{\uhp} (\dx \uhp)^2 - 2 \uhp |\dx \uhp|^2 \\
& = \uhp (\dx \uhp)^2 + 2 \ol{\uhp} |\dx \uhp|^2 + \uhp \ol{\dx \uhp}^2 + \ol{\uhp} \ol{\dx \uhp}^2
\end{aligned}
\end{align*}
are contained in $[\frac{N_v}{2^{2\delta}}, 2^{2\delta} N_v]$, then at least one of $\uhp$ on the right hand side is $(1-P_{\frac{N}{2^{\delta}} \le \cdot \le 2^{\delta} N}^+) \uhp_N$.
Hence, \eqref{est:uhperr} implies
\begin{align*}
\dot{\gamma}(t)
& = i \alpha |v|^{\frac{1}{4}} \int_{\R} \Big\{ 4 |\uhp|^2 \dx^2 \uhp + 2 (\uhp)^2 \ol{\dx^2 \uhp} + 3 \ol{\uhp} (\dx \uhp)^2 \\
& \hspace*{80pt} + 6 \uhp |\dx \uhp|^2 \Big\} \ol{\Psi}_v (t,x) dx + \err.
\end{align*}

We set $w_k (t,x) := e^{-i\phi (t,x)} \dx^k \uhp (t,x)$.
Owing to \eqref{est:u_infty0}, \eqref{ugsp1}, \eqref{eq:diffw}, Lemma \ref{lem:Xtilde} and \eqref{ugsl2}, \eqref{eq:diffwl2}, we obtain
\begin{align*}
& |v|^{\frac{1}{4}} \bigg| \int_{\R} |\uhp(t,x)|^2 \dx^2 \uhp(t,x) \ol{\Psi}_v (t,x) dx \\
& \hspace*{80pt} - \lambda^2 |\gamma (t,v)|^2 \int_{\R} \dx^2 \uhp(t,x)  \ol{\Psi}_v (t,x) dx \bigg| \\
& \lesssim |v|^{\frac{1}{4}} \left| \int_{\R} (|\uhp(t,x)|^2 - |\uhp (t,vt)|^2) \dx^2 \uhp(t,x)  \ol{\Psi}_v (t,x) dx \right| \\
& \quad + |v|^{\frac{1}{4}} \left| \int_{\R} (|\uhp (t,vt)|^2 - \lambda^2 |\gamma (t,v)|^2) \dx^2 \uhp(t,x)  \ol{\Psi}_v (t,x) dx \right| \\
& \lesssim t^{-1} \eps \bigg( \int_{\R} |w_0(t,x)-w_0(t,vt)| \chi (\lambda (x-vt)) dx  \\
& \hspace*{70pt} + t^{\frac{1}{5}} (t^{\frac{4}{5}} |v|)^{\frac{3}{8}} |\uhp (t,vt)- \lambda e^{i \phi (t,vt)} \gamma (t,v)| \bigg) \\
& = \err.
\end{align*}
Moreover, integration by parts and \eqref{intpsi} lead
\begin{align*}
& |v|^{\frac{1}{4}} \int_{\R} |\uhp(t,x)|^2 \dx^2 \uhp(t,x) \ol{\Psi}_v (t,x) dx \\
& = |v|^{\frac{1}{4}} \lambda^2 |\gamma (t,v)|^2 \int_{\R} \dx^2 \uhp(t,x)  \ol{\Psi}_v (t,x) dx + \err \\
& = -t^{-1} |\gamma (t,v)|^2 \gamma (t,v) + \err.
\end{align*}
From the same manner, we have
\begin{align*}
& |v|^{\frac{1}{4}} \int_{\R} (\uhp)^2 \ol{\dx^2 \uhp} \ol{\Psi}_v (t,x) dx= -t^{-1} |\gamma (t,v)|^2 \gamma (t,v) + \err, \\
& |v|^{\frac{1}{4}} \int_{\R} \ol{\uhp} (\dx \uhp)^2 \ol{\Psi}_v (t,x) dx= -t^{-1} |\gamma (t,v)|^2 \gamma (t,v) + \err, \\
& |v|^{\frac{1}{4}} \int_{\R} \uhp |\dx \uhp|^2 \ol{\Psi}_v (t,x) dx= t^{-1} |\gamma (t,v)|^2 \gamma (t,v) + \err,
\end{align*}
which imply that
\[
\dot{\gamma}(t) = -3i \alpha t^{-1} |\gamma (t,v)|^2 \gamma (t,v) + \err.
\]
\end{proof}

First, we show global existence of the solution to \eqref{5mKdV}.
From Proposition \ref{prop:WP} and Lemma \ref{lem:energy}, this is reduced to showing \eqref{est:u_infty}, that is to say, to close the bootstrap estimate \eqref{est:u_infty0}.
In the case $t^{-\frac{1}{5}} |x| \lesssim 1$, the bootstrap assumption \eqref{est:u_infty0}, Lemma \ref{lem:Xtilde}, and \eqref{est:uepp} yield
\[
\| \lr{t^{-\frac{1}{5}}x}^{-\frac{k}{4}+\frac{3}{8}} \dx^k u(t) \|_{L^{\infty}(t^{-\frac{1}{5}} |x| \lesssim 1)} \lesssim \eps t^{-\frac{k+1}{5}}.
\]
For the case $t^{-\frac{1}{5}} |x| \gtrsim 1$, owing to \eqref{ugsp1}, it is reduced to showing that
\[
\| \gamma (t) \|_{L^{\infty}_v(\Oo)} \lesssim \eps,
\]
where the implicit constant is independent of $D$ and $T$.

Let $C$ is a constant such that $v \in \Oo$ for $t \ge \max (1, C|v|^{-\frac{5}{4}})$.
For $|v| \ge C^{\frac{4}{5}}$, the Gagliardo-Nirenberg inequality, Proposition \ref{prop:WP} and Lemma \ref{lem:freq_psi} lead
\[
| \gamma (1,v) |
\lesssim \| \wh{u} (1) \|_{L^{\infty}}
= \| e^{\frac{1}{5}i\xi^5} \wh{u} (1) \|_{L^{\infty}}
\lesssim \| u(1) \|_{L^2}^{\frac{1}{2}} \| \J u(1) \|_{L^2}^{\frac{1}{2}}
\lesssim \eps.
\]
Solving the ordinary differential equation in Proposition \ref{prop:gamma_decay} with the initial time $t=1$, we have
\[
\gamma (t,v) = \gamma (1,v) e^{-3i\alpha |\gamma (1,v)|^2 \log t} + O \left( \eps (t^{\frac{4}{5}}|v|)^{-\frac{3}{16}} \right),
\]
which implies
\[
|\gamma (t,v) | \lesssim \eps
\]
for $t \in [1,T]$.

When $|v| < C^{\frac{4}{5}}$, let $t_0>1$ be $t_0 := C|v|^{-\frac{5}{4}}$.
Then, Bernstein's inequality, \eqref{Psi_L1}, and Lemma \ref{lem:Xtilde} yield
\[
|\gamma (t_0,v) |
\lesssim t_0^{\frac{1}{10}} \sum_{N \sim t_0^{-\frac{1}{5}}} \| u_N (t_0) \|_{L^2} + \eps
\lesssim \eps.
\]
Solving the ordinary differential equation in Proposition \ref{prop:gamma_decay} with the initial time $t=t_0$, we have
\[
\gamma (t,v) = \gamma (t_0,v) e^{-3i\alpha |\gamma (t_0,v)|^2 \log (t_0^{-1} t)} + O \left( \eps \right),
\]
which implies
\[
|\gamma (t,v) | \lesssim \eps
\]
for $t \in [t_0,T]$.
Accordingly, we conclude that \eqref{est:u_infty} holds for any $t \in [1,T]$.

Second, we focus on the asymptotic behavior of the global solution.
The estimates in the decaying region $\Xp$ follow from Lemma \ref{lem:Xtilde}, \eqref{uepL2}, and \eqref{est:uepp}.
Proposition \ref{prop:gamma_decay} yields that there exists a unique (complex valued) function $W$ defined on $(0,\infty)$ such that for $t \ge 1$,
\begin{equation} \label{gammaW}
\gamma (t,v) = \frac{1}{2} W(\xi_v) e^{-\frac{3}{4}i \alpha |W(\xi_v)|^2 \log (t |v|^{\frac{5}{4}})} + \wt{R}(t,v),
\end{equation}
where
\[
\| (t^{\frac{4}{5}} |v|)^{\frac{3}{16}} \wt{R} (t,v) \|_{L^{\infty} (\Oo)} + \| t^{\frac{2}{5}} \wt{R} (t,v) \|_{L^{\infty} (\Oo)} 
\lesssim \eps.
\]
We extend $W$ to $\R$ by defining
\[
W(-\xi) = \ol{W(\xi)}, \quad
W(0) = \int_{\R} u_0(x) dx.
\]
Then, we have
\[
\| W \|_{L^{\infty}} \le \eps.
\]
Proposition \ref{prop:approx} and \eqref{gammaW} show the estimates in $\Xn$.
Finally, we derive the asymptotic behavior in the self-similar region $\Xz$.
We use the self-similar change of variables \eqref{selfsimilar}.
Let $\rho >0$ be small specified later and let $0<C \ll 1$.
From \eqref{eq:U}, Bernstein's inequality, Lemmas \ref{lem:energy}, \ref{lem:Xtilde}, and \eqref{uepNl2}, we have
\begin{align*}
\| \dt P_{\le C t^{\rho}} U \|_{L^{\infty}_y (\Orz)}
& \lesssim t^{\frac{\rho}{2}} \| \dt P_{\le C t^{\rho}} U \|_{L^2_y (\Orz)} \\
& \lesssim t^{\frac{\rho}{2}} \| P_{\le C t^{\rho}} \dt U \|_{L^2_y (\Orz)} + t^{\frac{\rho}{2}-1} \| P_{\sim C t^{\rho}} U \|_{L^2_y (\Orz)} \\
& \lesssim t^{\frac{3}{2}\rho-\frac{11}{10}} \| \Lambda u \|_{L^2_x} + t^{\frac{\rho}{2}-\frac{9}{10}} \sum _{\substack{N \in 2^{\delta \mathbb{Z}} \\ N \sim C t^{\rho-\frac{1}{5}}}} \| \uep_N \|_{L^2} \\
& \lesssim \eps t^{-1-\min (-\frac{3}{2}\rho+\frac{1}{10}-\eps, \frac{7}{2}\rho)}.
\end{align*}
Furthermore, \eqref{est:ueperr}, \eqref{uepNl2}, and Lemma \ref{lem:Xtilde} yield
\begin{align*}
&
\begin{aligned}
& \| P_{> C t^{\rho}} U \|_{L^{\infty}_y (\Orz)} \\
& \lesssim t^{\frac{1}{5}} \Bigg( \sum _{\substack{N \in 2^{\delta \mathbb{Z}} \\ N > C t^{\rho-\frac{1}{5}}}} N \| \uep_N \|_{L^2}^2 \Bigg)^{\frac{1}{2}} + t^{\frac{1}{5}} \sum _{\substack{N \in 2^{\delta \mathbb{Z}} \\ N > C t^{\rho-\frac{1}{5}}}} \| (1-\PN) |\dx|^{\frac{1}{2}} \uep_N \|_{L^2} \\
& \lesssim t^{-\frac{7}{2} \rho} \eps,
\end{aligned}
\\
&
\begin{aligned}
& \| P_{> C t^{\rho}} U \|_{L^2_y (\Orz)} \\
& \lesssim t^{\frac{1}{10}} \Bigg( \sum _{\substack{N \in 2^{\delta \mathbb{Z}} \\ N > C t^{\rho-\frac{1}{5}}}} \| \uep_N \|_{L^2}^2 \Bigg)^{\frac{1}{2}} + t^{\frac{1}{10}} \sum _{\substack{N \in 2^{\delta \mathbb{Z}} \\ N > C t^{\rho-\frac{1}{5}}}} \| (1-\PN) \uep_N \|_{L^2} \\
& \lesssim t^{-4\rho} \eps.
\end{aligned}
\end{align*}
By setting $\rho := \frac{1}{5} (\frac{1}{10}-\eps)$, there exists $Q \in L^{\infty}_y (\R)$ such that
\[
\| Q \|_{L^{\infty}} \lesssim \eps, \quad
\| U - Q \|_{L^{\infty}_y (\Orz)} \lesssim \eps t^{-\frac{7}{2}\rho}, \quad
\| U - Q \|_{L^2_y (\Orz)} \lesssim \eps t^{-4\rho}.
\]
Because Lemma \ref{lem:Xtilde} implies
\[
\| \dy^4 U + y U + 5\alpha (2 U^2 \dy^2 U + 3 U (\dy U)^2) + 5\beta U^5 \|_{L^2_y}
= \| (\Lambda u) (t,t^{\frac{1}{5}}y) \|_{L^2_y}
\lesssim t^{\eps-\frac{1}{10}},
\]
taking the limit as $t \rightarrow \infty$, we have that $Q$ is a solution to \eqref{eq:self}.

\section*{Acknowledgment}

The author would like to thank Prof. Takamori Kato for fruitful discussions about this work.
This work was supported by JSPS KAKENHI Grant number JP16K17624 and Alumni Association ``Wakasatokai'' of Faculty of Engineering, Shinshu University.

\end{document}